\documentclass{article}
\usepackage[margin=1.5in]{geometry}

\usepackage{graphics}

\usepackage{algorithm}
\usepackage{algorithmic}
\usepackage{enumitem}

\usepackage{amsthm}
\newtheorem{theorem}{Theorem}
\newtheorem{lemma}{Lemma}
\newtheorem{corollary}{Corollary}
\newtheorem{proposition}{Proposition}
\newtheorem{definition}{Definition}
\newtheorem{remark}{Remark}
\newtheorem{example}{Example}

\newtheorem{model}{Model}

\usepackage{hyperref}
\usepackage{url}
\usepackage{amsmath}

\usepackage{amsfonts}
\usepackage{dsfont}
\usepackage{upgreek}
 \usepackage[bb=boondox]{mathalfa} 

\usepackage{graphicx}
\usepackage{epstopdf}
\graphicspath{ {./figures/} } 

\usepackage{caption}
\usepackage{subcaption}

\usepackage{xspace}

\newcommand{\Sec}[1]		{Sec.\,\ref{#1}}
\newcommand{\Fig}[1]		{Fig.\,\ref{#1}}

\newcommand{\Eq}[1]			{Eq.\,\ref{#1}}
\newcommand{\Tab}[1]		{Tab.\,\ref{#1}}

\newcommand{\Theorem}[1]{Theorem~\ref{#1}}
\newcommand{\Theorems}[1]{Theorems~\ref{#1}}
\newcommand{\Proposition}[1]{Proposition~\ref{#1}}

\newcommand{\Corollary}[1]{Corollary~\ref{#1}}
\newcommand{\Lemma}[1]{Lemma~\ref{#1}}

\newcommand{\Definition}[1]{Definition~\ref{#1}}

\newcommand{\ie}   			{i.e.~}
\newcommand{\eg}   			{e.g.~}

\newcommand{\st}   			{\mbox{s.t.~}}
\newcommand{\as}   			{\mbox{~a.s.}}
\newcommand{\aas}   		{\mbox{~a.a.s.}}
\newcommand{\wrt}   		{w.r.t.~}
\newcommand{\Erdos}   	{Erd\"os-R\'enyi }
\newcommand{\one}       {\mathds{1}}
\newcommand{\zero}      {\mathbb{0}}

\newcommand{\Exp}[1]    {\mathbb{E}(#1)}
\newcommand{\ExpBigP}[1]    {\mathbb{E}\left(#1\right)}
\newcommand{\Prob}[1]   {\mathbb{P}(#1)}
\newcommand{\ExpUnder}[2]    {\mathbb{E}_{#1}(#2)}
\newcommand{\ProbUnder}[2]   {\mathbb{P}_{#1}(#2)}
\newcommand{\op}[1]     {\,{#1}\,}
\newcommand{\card}[1] {\mathop{\rm card}(#1)}

\newcommand{\HazMat}   {\mathcal{H}}
\newcommand{\HazSpec}  {\rho_n}
\newcommand{\AdjMat}   {\mathcal{A}}
\newcommand{\RandomG}   {G}
\newcommand{\OneToN}   {[|n|]}

\title{Spectral Bounds in Random Graphs Applied to Spreading Phenomena and Percolation}

\author{
R\'emi Lemonnier$^{1,2}$
\hspace{1.5em}
Kevin Scaman$^1$
\hspace{1.5em}
Nicolas Vayatis$^1$
\\$^1$ CMLA {--} ENS Cachan, CNRS, Universit\'e Paris-Saclay, France
\\$^2$ Numberly, 1000Mercis group, Paris, France
\\\texttt{\{lemonnier, scaman, vayatis\}@cmla.ens-cachan.fr}
}
\date{}

\begin{document}

\maketitle

\begin{abstract}
In this paper, we derive nonasymptotic theoretical bounds for the influence in random graphs that depend on the spectral radius of a particular matrix, called the \emph{Hazard matrix}. We also show that these results are generic and valid for a large class of random graphs displaying correlation at a local scale, called the LPC random graphs. In particular, they lead to tight and novel bounds in percolation, epidemiology and information cascades.
The main result of the paper states that the influence in the sub-critical regime for LPC random graphs is at most of the order of $O(\sqrt{n})$ where $n$ is the size of the network, and of $O(n^{2/3})$ in the critical regime, where the epidemic thresholds are driven by the size of the spectral radius of the Hazard matrix with respect to 1.
 As a corollary,  it is also shown that such bounds hold for the size of the giant component in inhomogeneous percolation, the SIR model in epidemiology, as well as for the long-term influence of a node in the Independent Cascade Model.
\end{abstract}

\section{Introduction}
Propagation models over graphs are very popular and particularly well suited to the analysis of epidemics and information cascades. Although different in many technical aspects, the models used in these two fields are similar and can be considered as particular instances of a more generic framework: the analysis of the influence of reachable sets in random networks.

In epidemiology, the study of diffusion models such as SI, SIS or SIR \cite{Newman:2010:NI,kermack1932contributions} highlighted the impact of a spectral characteristic on the size of the epidemic: the spectral radius of the underlying network. Moreover, it was shown that this quantity acted as a critical threshold for the size of the epidemic \cite{van2009virus,prakash2012threshold}, and recent work provided upper bounds that depend highly on this spectral quantity \cite{draief2008}. Our work can be seen as a generalization of these works, by providing the right spectral quantity to consider in the case of more generic diffusion and percolation models.

 In percolation theory, the concept of reachability characterizes the connected components of undirected graphs and the behavior of such components has been the object of several studies. For homogeneous random graphs $\RandomG(n,p)$ where removal of edges in the fully connected graph with $n$ vertices occurs independently for every  edge with constant probability $1-p$, Erd\"os and Renyi \cite{erd6s1960evolution} showed that a phase transition occured for $p=\frac{1}{n}$, and their results were later refined by Bollob\'as \cite{bollobas1984evolution} and Lucszak \cite{luczak1990component} for the case $pn=O(1)$.
 For inhomogeneous graphs, we refer to the work by Bollob\'as, Janson and Riordan \cite{bollobas2007phase} in the special case where the number of edges $E$ is $O(n)$, and Bollob\'as, Borgs, Chayes and Riordan \cite{bollobas2010percolation} when $E=O(n^2)$. These references contain a number of asymptotic results (\emph{i.e.} when $n \rightarrow \infty$) including the critical value of the percolation threshold, and upper bounds on the size of connected components.

In the present work, we introduce the notion of random graphs with \emph{Local Positive Correlation} (LPC) (see Definition \ref{def:LPC}) and  derive nonasymptotic upper bounds for the influence in this setup. The concept of random graphs with LPC unifies,  in some sense, the description of the phenomena observed in the fields of percolation theory, epidemiology and information cascades. The upper bounds obtained depend on the spectral radius $\HazSpec$ of a particular matrix built from the edge probabilities, called the \emph{Hazard matrix}. We show that such bounds reveal three regimes: subcritical, critical and supercritical, depending on the value of the spectral radius $\HazSpec$. For random graphs with $n$ vertices, we show that  the influence is at most a $O(\sqrt{n})$ when $\HazSpec < 1$, and in average a $O(1)$. However, when $\HazSpec > 1$, the regime becomes supercritical as the influence becomes potentially linear in $n$. More specifically, we show that the influence is upper bounded by $\gamma_0(\HazSpec) n + o(n)$, where $\gamma_0(\HazSpec)\in[0,1]$ is a simple function (see \Definition{def:gamma}) and that this bound is met for particular random graphs. Finally, in the transitional regime where $\HazSpec \approx 1$, the influence is at most a $O(n^{2/3})$, and in average a $O(\sqrt{n})$. Moreover, we also obtain that the size of this intermediate regime \wrt $\HazSpec$ is proportional to $n^{-1/3}$. 
\Tab{tab:summary} summarizes the different behaviors of upper bounds for influence in random graphs with LPC, in the subcritical, critical and supercritical regimes, as provided in \Sec{sec:bounds}. In the \emph{Random A} scenario, a set of $n_0$ influencers are drawn at random, while in \emph{Random B} each node belongs to the influencer set with independent probability $q$.

\renewcommand{\arraystretch}{1.3}
\begin{table}[h]
\centering
\begin{tabular}{|r|c|c|c|}
\hline
 & \multicolumn{3}{ c| }{\textbf{Regimes}}\\
\textbf{Scenario} & Subcritical ($\HazSpec < 1$) & Critical ($\HazSpec \approx 1$) & Supercritical ($\HazSpec > 1$)\\
\hline\hline
(I) Worst-case & $O(\sqrt{n})$ & $O(n^{2/3})$ & $\gamma_0(\HazSpec)n + O(\sqrt{n})$ \\
\hline
(II) Random A & $O(1)$ & $O(\sqrt{n})$ & $\gamma_0(\HazSpec)n + O(1)$ \\
\hline
(III) Random B & $O(qn)$ & $O(\sqrt{q}n)$ & $\gamma_0(\HazSpec)n + O(qn)$ \\
\hline
\end{tabular}
\caption{Summary of results for influence in random graphs with LPC.}
\label{tab:summary}
\end{table}

As a corollary, we derive upper bounds for the size of the giant component in bond and site percolation which significantly improve the previous results of \cite{bollobas2010percolation}. More specifically, we show that the spectral radius $\HazSpec$ is a key quantity for percolation, and that the size of the giant component $C_1(\RandomG)$ is, in expectation, upper bounded by a $O(\sqrt{n})$ when $\HazSpec < 1$, by a $O(n^{2/3})$ when $|\HazSpec - 1| = O(n^{-1/3})$, and by $\gamma_0(\HazSpec) n + o(n)$ when $\HazSpec > 1$. Moreover, we prove that a giant component can only exist if $\limsup_{n\rightarrow +\infty} \HazSpec > 1$. Also, we analyze the distribution of the size of connected components by upper bounding the number $N(m)$ of connected components of size bigger than $m$ in expectation.
\Tab{tab:summaryPerco} summarizes the different behaviors of the upper bounds in the subcritical, critical and supercritical regimes, derived  for percolation.

\renewcommand{\arraystretch}{1.3}
\begin{table}[h]
\centering
\begin{tabular}{|r|c|c|c|}
\hline
 & \multicolumn{3}{ c| }{\textbf{Regimes}}\\
\textbf{Quantity} & Subcritical ($\HazSpec < 1$) & Critical ($\HazSpec \approx 1$) & Supercritical ($\HazSpec > 1$)\\
\hline\hline
$\Exp{C_1(\RandomG)}$ & $O(\sqrt{n})$ & $O(n^{2/3})$ & $\gamma_0(\HazSpec)n + O(\sqrt{n})$ \\
\hline
$\Exp{N(m)}$ & $O(nm^{-2})$ & $O(nm^{-3/2})$ & $\gamma_0(\HazSpec)n/m + O(nm^{-3/2})$ \\
\hline
\end{tabular}
\caption{Summary of results for bond and site percolation: $C_1(G)$ is the size of the giant component, and $N(m)$ is the number of connected components of size bigger than $m$.}
\label{tab:summaryPerco}
\end{table}

Finally, we apply our upper bounds to the late-time properties of the Susceptible-Infected-Removed (SIR) epidemic model, as well as discrete and continuous-time information cascades. More specifically, we significantly improve the results of \cite{draief2008} in the subcritical regime, and show that, near the epidemic threshold, the number of infected nodes in the SIR model is a $O(n^{2/3})$. Furthermore, we extend the traditional epidemic threshold in $\beta\rho(\AdjMat) = \delta$, where $\beta$ and $\delta$ are the transmission and recovery rates and $\AdjMat$ is the adjacency matrix of the underlying graph, to more realistic SIR models in which the incubation period may follow a non-exponential distribution. 

The remainder of the paper is organized as follows. In \Sec{sec:model}, we recall the notions of reachable set and influence in random networks, and introduce a generic type of random graphs with \emph{Local Positive Correlation} (LPC). In \Sec{sec:bounds}, we derive theoretical bounds for the influence in random graphs with LPC. Finally, in \Sec{sec:percolation}, \Sec{sec:sitePerco}, \Sec{sec:epidemiology} and \Sec{sec:IC}, we show that the previous results apply respectively to the fields of bond percolation, siet percolation, epidemiology and information cascades, and improve existing results in these fields.

\section{Random graphs, Hazard matrix, influence\\and LPC property}\label{sec:model}
In this section, we introduce the main   notations and definitions. In particular, we define two novel concepts: the \emph{Hazard matrix}, that will play a key role in the analysis of influence in random graphs, and a generic class of random graphs with \emph{Local Positive Correlation} (LPC).

\subsection{Setup}
We now provide useful notations and a precise definition of random graphs used thereafter.\\

\noindent {\em General notations.} For any set $X$, we will denote as $\card{X}$ its number of elements, $\mathcal{P}_n(X)$ the set of all subsets of $X$ of size $n$ and $X\setminus Y$ the complementary subset of $Y$ in $X$. We will also use the abbreviation $\OneToN = \{1,...,n\}$ the set of all integers between $1$ and $n$, and $\one\{\cdot\}$ the indicator function.
We will say that a property $A$ holds \emph{almost surely} (abbreviated as $\as$) if $\Prob{A} = 1$, and that a sequence of properties $A_n$ holds \emph{asymptotically almost surely} (abbreviated as $\aas$) if $\lim_{n\rightarrow +\infty}\Prob{A_n} = 1$.

\medskip

\noindent {\em Random graphs.} Let $n>0$ be a fixed integer. We consider the set of all graphs  $\mathcal{G} = (\mathcal{V}, \mathcal{E})$ with labelled vertices $\mathcal{V}=\OneToN$ and edge set $\vec{\mathcal{E}} \subset \OneToN^2$.
A random graph of size $n$ is a random element in this set of all possible graphs. Such a random graph is entirely characterized by its \emph{random} adjacency matrix $A\in\{0,1\}^{n^2}$, defined by $A_{ij} = 1$ if $(i,j) \in\vec{\mathcal{E}}$, else $A_{ij} = 0$.
In what follows, we will use the notation $\RandomG(n, A)$ to denote a random graph of size $n$ and random adjacency matrix $A\in\{0,1\}^{n^2}$, where  $A_{ij}$ are Bernoulli random variables indicating the presence or absence of edge $(i,j)$ in the random graph.
We will call \emph{undirected} a random graph whose adjacency matrix is \emph{symmetric}, \ie $\forall i,j$, $A_{ij} = A_{ji} \as$. The simplest example of undirected random graph is the \Erdos random graph $\RandomG(n,p)$ whose adjacency matrix has independent and identically distributed (i.i.d.) edge presence variables $\{A_{ij} : i < j\}$ and $\Exp{A_{ij}} = p$.
Note that, in general, the edge variables $A_{ij}$ are correlated.

\subsection{Hazard characteristics of random graphs}
For many diffusion models, the spectral features of the underlying graph were shown to have a drastic impact on the amplitude of the spread (see for example the role of the spectral radius of the adjacency matrix in the epidemiology literature \cite{van2009virus,prakash2012threshold}). In order to generalize such results to a broader class of diffusion and percolation phenomena, we introduce two spectral characteristics that are better suited to the analysis of the influence in random graphs: the \emph{Hazard matrix} and the \emph{Hazard radius}. To the best of our knowledge, these concepts have not been considered before (besides our preliminary results presented recently \cite{NIPS2014_5364,NIPS2015_5701}).

\begin{definition}[\textbf{Hazard matrix}]\label{def:hazardMatrix}
For a  random graph model $\RandomG (n, A)$, the \emph{Hazard matrix}  $\HazMat$ is the $n\times n$ matrix whose coefficients $\HazMat_{ij}$ are defined as:
\begin{equation}
\HazMat_{ij} = -\ln(1-\Exp{A_{ij}})~.
\end{equation}
\end{definition}

The spectral radius of this matrix will play a key role in the quantification of the influence. We recall that for any square matrix $M$ of size $n$, its spectral radius $\rho(M)$ is defined as the largest of the  the eigenvalues of  $M$.

\begin{definition}[\textbf{Hazard radius}]
For a  random graph model $\RandomG  (n, A)$ with \emph{Hazard matrix} $\HazMat$, we define the \emph{Hazard radius} as:
\begin{equation}
\HazSpec =\rho \left(\frac{\HazMat+\HazMat^\top}{2}\right)~.
\end{equation}
\end{definition}

\begin{remark}\label{rem:pij}
Let $P = (\Exp{A_{ij}})_{ij}$ be the expected adjacency matrix. When the $P_{ij}$'s are small, the Hazard matrix is very close to $P$. This implies that, for small values of $P_{ij}$, the spectral radius of $\HazMat$ will be very close to that of $P$. More specifically, a simple calculation holds
\begin{equation}
\rho(P) \leq \rho(\HazMat) \leq \frac{-\ln(1-\|P\|_\infty)}{\|P\|_\infty} \rho(P),
\end{equation}
where $\|P\|_\infty = \max_{i,j} P_{ij}$. The relatively slow increase of $\frac{-\ln(1-x)}{x}$ for $x\rightarrow 1^-$ implies that the behavior of $\rho(P)$ and $\rho(\HazMat)$ will be of the same order of magnitude even for large (but lower than $1$) values of $\|P\|_\infty$. Moreover, when considering a sequence of random graphs $\RandomG_n$, if $\lim_{n\rightarrow +\infty} \|P_n\|_\infty = 0$, then $\rho(\HazMat_n) \approx \rho(P_n)$ and the subcriticality of the influence is also induced by $\limsup_{n\rightarrow +\infty} \rho(\frac{P_n + P_n^\top}{2}) \leq 1$ (see \Sec{sec:bounds}).
\end{remark}

In addition, we introduce here a useful function that will allow the simplification of the upper bounds derived in this paper.

\begin{definition}[\textbf{Hazard function}]\label{def:gamma}
let $\rho \geq 0$ and $a > 0$. The  \emph{Hazard function} $\gamma(\rho, a)$ is defined as the unique solution in $[0, 1]$ of the following equation:
\begin{equation}
\gamma - 1 + \exp\left(-\rho\gamma - a\right) = 0~.
\end{equation}
We will also use the notation $\gamma_0(\rho) = \lim_{a\rightarrow 0^+} \gamma(\rho, a)$ for the limit of the Hazard function at 0.
\end{definition}
\Fig{fig:gamma} reveals the behavior of $\gamma_0(\rho)$ and $\gamma(\rho, a)$ \wrt $\rho$.
For more information on the Hazard function and the bounds used to derive the subcritical, critical and supercritical regimes, we refer to Appendix \ref{appendix:gamma}.

\begin{figure}[h]
	\centering
		\includegraphics[viewport=16 0 235 172, clip=true, width=0.6\linewidth]{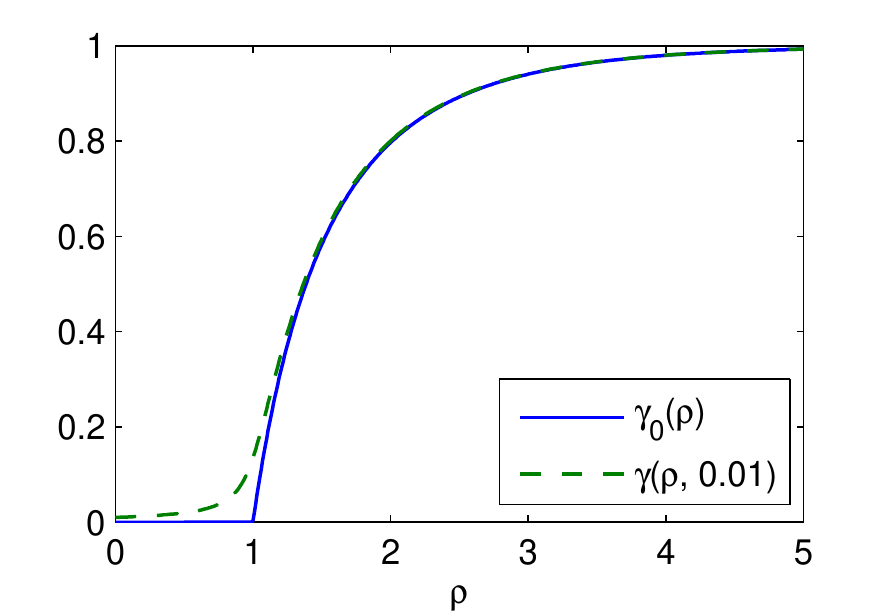}
		\caption{Behavior of $\gamma_0(\rho)$ and $\gamma(\rho, a)$ \wrt $\rho$.}
		\label{fig:gamma}
\end{figure}

\subsection{Reachability and influence}
In this section, we define the \emph{influence} as the size of a \emph{reachable set}. A node $u$ is \emph{reachable} from another node $v$ if there is a path connecting  $u$ to $v$ in the considered graph \cite{opac-b1130933}. As we will see later, this definition generalizes the notion of influence in information cascades \cite{Kempe:2003:MSI:956750.956769}, the size of a contagion in epidemiology and the size of a connected component in percolation.

\begin{definition}[\textbf{Reachable set}]\label{def:reachableSet}
Consider a random graph $\RandomG (n, A)$. We call influencers a fixed set $\mathcal{I} \subset \OneToN$ of nodes and we define the \emph{reachable set} of influencers $\mathcal{I}$ in $\mathcal{G}$ the random set of nodes $R(\mathcal{I}, A)$ such that:
\begin{equation}\label{pathdef}
i\in R(\mathcal{I}, A) \iff i \in \mathcal{I} \mbox{ or } \prod_{q\in\mathcal{Q}_{\mathcal{I},i}}\left(1 - \prod_{(j,k)\in q}A_{jk}\right) = 0~.
\end{equation}
where, for any node $i \in \OneToN$, the collection $\mathcal{Q}_{\mathcal{I},i} = \{\{(i_0,i_1),(i_1,i_2),...,(i_{k-1},i_k)\} : k\in\mathbb{N}, i_0\in\mathcal{I}, i_k = i \mbox{ and all } i_j \mbox{ are distinct}\}$ is the set of directed paths (removing the loops) from the set $\mathcal{I}$ to node $i$.
\end{definition}

Informally, a node $i$ belongs to the reachable set of $\mathcal{I}$ if and only if there is a path from $\mathcal{I}$ to $i$ in the graph. By extension, we will call the \emph{reachable set of node $i$} the reachable set of $\{i\}$. This definition will be used to characterize the asymptotic behavior of the state vector in a contagion process over a graph, as well as connected components of random undirected networks. 
As we will see in sections \ref{sec:percolation}, \ref{sec:epidemiology} and \ref{sec:IC}, this setting is general enough to include many reference models used in the fields of percolation theory, epidemiology and information cascades.

\medskip

We now introduce the notion of \emph{influence} of a set $\mathcal{I}$ of nodes, denoted as $\sigma(\mathcal{I})$, as the expected size of the reachable set of $\mathcal{I}$ with respect to the random graph model $\RandomG(n, A)$.

\begin{definition}[\textbf{Influence}]
Given a random graph model $\RandomG(n, A)$ and a fixed {\em set of influencers} $\mathcal{I} \subset \OneToN$, the influence of $\mathcal{I}$ in $\RandomG(n, A)$ is defined as the quantity:
\begin{equation}
\sigma(\mathcal{I}) = \Exp{\card{R(\mathcal{I}, A)}}~,
\end{equation}
where $R(\mathcal{I}, A)$ is the reachable set of $\mathcal{I}$ in $\RandomG (n, A)$.
\end{definition}

\medskip

\noindent {\em Examples.} In order to illustrate the previous concepts, we now relate Hazard radiuses to critical thresholds for influence in four particular random graphs. In the first two, we will show that the critical threshold for influence can be restated as $\HazSpec = 1$. The other two examples are cases in which the threshold may differ, sometimes significantly, from $\HazSpec = 1$.

\begin{example}[\textbf{\Erdos random graphs}]
For \Erdos random graphs $\RandomG(n,\frac{c}{n})$, whose adjacency matrix has i.i.d. edge  variables $\{A_{ij} : i \leq j\}$ and $\Exp{A_{ij}} = \frac{c}{n}$, percolation theory  (\cite{erd6s1960evolution}) states that a threshold phenomenon occurs for $c = 1$. Moreover, using \Definition{def:hazardMatrix}, $\HazMat^{n}_{ij} = -\ln(1-\frac{c}{n})$ and we have:
\begin{equation}
\HazSpec = -n\ln\left(1-\frac{c}{n}\right)~.
\end{equation}
Hence, for $\RandomG(n, \frac{c}{n})$, criticality arises when $\HazSpec \rightarrow 1$ as $n$ tends to infinity.
\end{example}

\begin{example}[\textbf{Poissonian graph processes}]
We now consider a particular random graph, called the \emph{Poissonian graph process or Norros-Reittu model} (\cite{norros2006}) and closely related to random graphs of fixed degree distribution known as the \emph{configuration model} (\cite{molloy1995critical,molloy1998size}). More specifically, let $w = (w_i)_{i\in\OneToN}$ be a weight vector, and $\RandomG(n,w)$ an undirected random graph of $n$ nodes and adjacency matrix $A$, where, for $i \leq j$, $A_{ij}$ are independent Bernouilli random variables of parameter 
$P_{ij} = 1 - \exp\left(-\frac{w_i w_j}{\sum_k w_k}\right)$.
Note that self-loops are allowed, but they hardly occur when the weight distribution is close to uniform. Such a random graph has a Hazard radius equal to
\begin{equation}
\HazSpec = \frac{\sum_i w_i^2}{\sum_k w_k}.
\end{equation}
Previous results \cite{bollobas2007phase, EJP817} showed that, for such graphs, a giant component exists if and only if $\sum_i{w_i^2} > \sum_i{w_i}$, which is equivalent to $\HazSpec > 1$.
\end{example}

\begin{example}[\textbf{Homogeneous percolation on regular grids}]
Let $\mathcal{G} = (\mathcal{V},\mathcal{E})$ be a regular cubic grid of $n$ nodes in dimension $d$, and $\AdjMat$ its adjacency matrix. The random graph $\RandomG(n,A)$ is the result of homogeneous percolation on $\mathcal{G}$ if $\{A_{ij} : \{i,j\}\in\mathcal{E}\}$ are i.i.d. Bernoulli random variables of fixed parameter $p\in[0,1]$, and $A_{ij} = 0$ otherwise. The Hazard radius of such a network is $\HazSpec = -\rho(\AdjMat)\ln(1-p) \rightarrow_{n\rightarrow +\infty} -2d\ln(1-p)$. For $d>2$, there are no known exact formula for percolation thresholds on cubic grids, although experimental approaches provided high precision numerical estimates \cite{grassberger2003critical}. These estimates seem to coincide rather well with $p = 1-e^{-1/2d}$ (\ie the value such that $\lim_{n\rightarrow +\infty}\HazSpec = 1$) for high-dimensional grids (for $d=13$, $p^* = 0.040$ compared to $1-e^{-1/2d} = 0.038$), while being rather different for lower-dimensional grids (for $d=2$, $p^* = 0.5$ compared to $1-e^{-1/2d} = 0.22$).
\end{example}

\begin{example}[\textbf{Star-shaped network}]
For homogeneous percolation on a star-shaped network centered around $1$, the exact influence of $\mathcal{I} = \{1\}$ can be derived explicitly and we have: $\sigma(\{1\}) = 1+p(n-1)$.
As, for $i<j$, the Hazard matrix coefficients are $\HazMat_{ij} = -\ln(1-p) \one\{i = 1\}$, the Hazard radius is given by $\HazSpec = -\sqrt{n-1} \ln(1-p)$.
When $p = c/\sqrt{n}$, the influence is always sublinear in $n$, and the threshold value is infinite ($c^* = +\infty$). Hence $\lim_{n\rightarrow +\infty}\HazSpec = c = 1$ does not bring any particular change in the behavior of the influence.
\end{example}

\subsection{Random graphs with Local Positive Correlation (LPC)}

The analysis developed in this paper concerns a particular class of random graphs that display correlation at a local scale only.

\begin{definition}[\textbf{Random graphs with Local Positive Correlation (LPC)}]
\label{def:LPC}
We consider a random graph $\RandomG(n, A)$. For any node pair $(i,j)\in\OneToN^2$, we define $A_{-ij}$ to be the subcollection of edge variables $\{A_{kl} ~:~ (k,l)\neq(i,j) \}$, and $\mathcal{N}_{ij} = \{(j,i)\}\cup\{(k,l) : k = i \mbox{ or } l = j\}$ to be the neighborhood that contains the edge $(j,i)$ plus all edges which share with $(i,j)$ either the same head or the same tail. 
The random graph $\RandomG (n, A)$ is a random graph with \emph{Local Positive Correlation} (LPC) if the two following conditions hold:
\begin{enumerate}
\item[(H1)]$\forall i, j, k, l$ such that $(k,l)\notin\mathcal{N}_{ij}$, $A_{ij}$ is independent of $A_{kl}$.
\item[(H2)]$\forall i, j$, the mapping $a \mapsto \Exp{A_{ij} | A_{-ij} \op{=} a}$ is non-decreasing \wrt the natural partial order on $\{0,1\}^{n^2-1}$ (\ie $a\leq a'$ if and only if $\forall i\leq n^2-1, a_i\leq a_i'$).
\end{enumerate}
\end{definition}
The first assumption (H1) can be interpreted as a property of long range independence between remote edges, while the assumption (H2) properly states the  local positive correlation of neigboring edges.
When the random variables $A_{ij}$ are interpreted as indicator variables of the occurrence of transmission events from node $i$ to node $j$ in a diffusion process, then the long range independence assumption implies pairwise independence for transmission events on nonadjacent edges, and the local positive correlation sets positive correlations at the local level for the transmission events on adjacent edges conditionally to the state of all other edges. 

\begin{remark}
Since independence implies positive correlation, all random graphs which assume independence of the edge variables also verify LPC. Hence, most standard models of random networks, including \Erdos and the very general class of \emph{inhomogeneous random graphs} \cite{bollobas2007phase}, verify LPC.
\end{remark}

\medskip

The following lemma indicates that the notion of random graphs with LPC covers in particular homogeneous and inhomogeneous percolation.

\begin{lemma}[\textbf{Random undirected graph}]\label{lem:undirectedLPC}
An undirected random graph $\RandomG (n, A)$ is a random graph with LPC if and only if the edge variables $\{A_{ij} : i < j\}$ are independent.
\end{lemma}
\begin{proof}
Since independence implies positive correlation, the LPC property is a direct consequence of $\{A_{ij} : i < j\}$ being independent. Now, if an undirected random graph $\RandomG(n,A)$ is a random graph with LPC, then, since $A_{ij}=A_{ji}$, (H1) holds for all $(k,l)\notin\mathcal{N}_{ij}\cap\mathcal{N}_{ji} = \{(i,j),(j,i)\}$, and $\{A_{ij} : i < j\}$ are independent.
\end{proof}

We will see later that the conditions of LPC are fulfilled for many popular random graph models used in epidemiology (\Sec{sec:epidemiology}) and in information propagation (\Sec{sec:IC}).

\section{Nonasymptotic upper bounds on the influence}\label{sec:bounds}
In this section, we provide tight upper bounds on the influence under three scenarios on the set of influencers:
\begin{itemize}
\item[I)] Deterministic (worst-case) scenario: corresponds to a fixed set $\mathcal{I}$ of influencers,
\item[II)] Random A with parameter $n_0<n$: the set $\mathcal{I}$ of influencers is random and drawn according to a uniform distribution $\mathcal{U}(\OneToN, n_0)$ over the subsets of $\OneToN$ of cardinality $n_0$,
\item[III)] Random B with parameter $q \in [0,1]$: the set $\mathcal{I}$ of influencers is random and drawn with a distribution $\mathcal{D}(\OneToN, q)$ such that, for all $i \in \OneToN$, the  random variables $B_i = \one \{ i \in \mathcal{I}\}$ are independent Bernoulli with parameter $q$.
\end{itemize}
In both cases (II) and (III), the influencer set is drawn independently of the particular graph sampled under the random graph model $\RandomG  (n, A)$.

\subsection{Deterministic (worst-case) scenario}
The first result (\Theorem{th:mainResult}) applies to any fixed set of influencers $\mathcal{I}$ such that $\card{\mathcal{I}}=n_0$. Intuitively, this result corresponds to a worst-case scenario since the bound does not depend on $\mathcal{I}$.

\begin{theorem}\label{th:mainResult}
Let $\RandomG  (n, A)$ be a random graph with LPC. For any fixed set of influencers $\mathcal{I}$ such that $\card{\mathcal{I}}=n_0 \leq n$, the influence $\sigma(\mathcal{I})$ is upper bounded by:
\begin{equation}
\sigma(\mathcal{I}) \leq n_0+ \gamma_1\left(\HazSpec,\frac{n_0}{n-n_0}\right) (n-n_0)~,
\end{equation}
where $\gamma_1(\rho,a) = \gamma_1$ is the smallest solution in $[0, 1]$ of the following equation:
\begin{equation}
\gamma_1 - 1 + \exp \left(-\rho \gamma_1 - \frac {\rho a}{\gamma_1}\right) = 0~,
\end{equation}
and $\HazSpec$ is the Hazard radius of the random graph $\RandomG(n, A)$.
\end{theorem}

A refined analysis of the parameter $\gamma_1$ in the previous theorem leads to the description of the behavior of the influence with respect to three regimes as shown in the following corollary.

\begin{corollary}\label{cor:simpleBounds}
Under the same conditions as in \Theorem{th:mainResult}, we have:
\[
\displaystyle
\sigma(\mathcal{I}) \leq
\left\{
\begin{array}{ll}
 \displaystyle n_0+\sqrt{\frac{\HazSpec}{1-\HazSpec}} \sqrt{n_0(n-n_0)}~, & \text{if } \HazSpec < 1 - \delta_n \\
 &\\
  \displaystyle n_0+ 2^{4/3} n_0^{1/3} (n - n_0)^{2/3}~, &\text{if }  |\HazSpec - 1| \leq \delta_n \\
  &\\
  \displaystyle n_0 + (n-n_0)\gamma_0(\HazSpec) + c_n \sqrt{n_0(n-n_0)}~,  & \text{if } \HazSpec > 1 + \delta_n
\end{array}
\right.
\]
where  $\delta_n =\displaystyle \left(\frac{n_0}{4(n - n_0)}\right)^{1/3} $, and $c_n = \displaystyle \sqrt{\frac{(1 - \gamma_0(\HazSpec))\HazSpec}{1 - (1-\gamma_0(\HazSpec))\HazSpec}}$.
\end{corollary}

\begin{remark}
When considering a fixed set of influencers $\mathcal{I}$, one can remove the ingoing edges of the influencers in order to get slightly improved results. In such a case, the Hazard matrix is replaced by $\HazMat_{ij}(\mathcal{I}) = \one\{j\notin \mathcal{I}\}\cdot \HazMat_{ij}$.
\end{remark}

\subsection{Random influencer set with fixed size}
The second result (\Theorem{th:uniformResult}) applies in the case where the set of influencers is drawn from a uniform distribution over the subpartition of sets of $n_0$ nodes chosen amongst $n$. This result corresponds to the average-case scenario in a setting where the initial influencer nodes are not known and drawn independently of the random graph.

\begin{theorem}
\label{th:uniformResult}
Let $\RandomG (n, A)$ be a random graph with LPC. Assume the set $\mathcal{I}$ of influencers  is  drawn from an independent and uniform distribution $\mathcal{U}(\OneToN, n_0)$ over $\mathcal{P}_{n_0}(\OneToN)$. Then, we have the following result:
\begin{equation}
\Exp{\sigma(\mathcal{I})} \leq n_0 + \gamma\left(\HazSpec, \frac{ n_0 \HazSpec}{n-n_0}\right) (n-n_0)~,
\end{equation}
where $\gamma(\rho,a)$ is defined as in \Definition{def:gamma}.
\end{theorem}

\begin{corollary}\label{cor:uniformsimpleBounds}
Under the same conditions as in \Theorem{th:uniformResult}, we have:
\[
\displaystyle
\Exp{\sigma(\mathcal{I})} \leq
\left\{
\begin{array}{ll}
 \displaystyle \frac{n_0}{1-\HazSpec}~,  & \text{if } \HazSpec < 1 - \delta'_n \\
 &\\
  \displaystyle n_0 + \sqrt{8 n_0(n - n_0)}~, &\text{if }  |\HazSpec - 1| \leq \delta'_n \\
  &\\
  \displaystyle (n-n_0)\gamma_0(\HazSpec) + \frac{n_0}{1 - \HazSpec(1-\gamma_0(\HazSpec))}~, & \text{if } \HazSpec > 1 + \delta'_n
\end{array}
\right.
\]
where  $\displaystyle \delta'_n = \sqrt{\frac{n_0}{2(n - n_0)}}$.
\end{corollary}

\subsection{Random influencer set with random size}
The third result (\Theorem{th:randomResult}) applies in the case where each node of the network is an influencer independently and with a fixed probability $q\in[0,1]$. This result corresponds to the randomized scenario in which the initial influencer nodes are not known and drawn independently of the random graph.

\begin{theorem}\label{th:randomResult}
Let $\RandomG(n, A)$ be a random graph with LPC. Assume each node is an influencer with independent probability $q\in[0,1]$ and denote by $\mathcal{I}\sim \mathcal{D}(\OneToN, q)$ the random set of influencers that is drawn. Then, we have the following result:
\begin{equation}
\Exp{\sigma(\mathcal{I})} \leq \gamma(\HazSpec, -\ln(1-q)) n,
\end{equation}
where $\gamma(\rho,a)$ is defined as in \Definition{def:gamma}.
\end{theorem}

\begin{corollary}\label{cor:randomsimpleBounds}
Under the same conditions as in \Theorem{th:randomResult}, we have:
\[
\displaystyle
\Exp{\sigma(\mathcal{I})} \leq
\left\{
\begin{array}{ll}
 \displaystyle \frac{-\ln(1-q)n}{1-\HazSpec}~,  & \text{if } \HazSpec < 1 - d_q \\
 &\\
  \displaystyle n\sqrt{-8\ln(1-q)}~, &\text{if }  |\HazSpec - 1| \leq d_q \\
  &\\
  \displaystyle n\gamma_0(\HazSpec) + \frac{-\ln(1-q)(1-\gamma_0(\HazSpec))n}{1 - \HazSpec(1-\gamma_0(\HazSpec))}~, & \text{if } \HazSpec > 1 + d_q
\end{array}
\right.
\]
where  $d_q = \displaystyle \sqrt{\frac{-\ln(1-q)}{2}}$.
\end{corollary}

\subsection{Lower bounds}
The following proposition shows that the upper bounds in \Corollary{cor:simpleBounds} are tight, in the sense that, for any Hazard radius, there is a random graph on which the influence has the exact same behavior as the upper bounds in the worst case scenario.
\begin{proposition}\label{prop:tightness}
For all $\rho > 0$, there exists a constant $C_\rho > 0$ and a sequence of LPC random graphs $(\RandomG_n)_{n > 0}$ with $n$ vertices and Hazard radius $\HazSpec$, and such that $\lim_{n\rightarrow +\infty} \HazSpec = \rho$. For $n$ sufficiently large, the influence $\sigma_n(\{1\})$ of node $1$ in $\RandomG_n$ is lower bounded by:
\begin{equation}
\sigma_n(\{1\}) \geq \left\{
\begin{array}{ll}
C_\rho \sqrt{n} &\mbox{ if } \rho < 1\\
C_\rho n^{2/3} &\mbox{ if } \rho = 1\\
\gamma_0(\rho)n - o(n) &\mbox{ if } \rho > 1
\end{array}\right..
\end{equation}
\end{proposition}
This proposition relies on ``\emph{random star-networks}'', \ie undirected random graphs such that the  $\{A_{ij} : i<j\}$ are independent Bernoulli random variables of parameter $a\in[0,1]$ if $i = 1$ and $b<a$ otherwise. Intuitively, such a network is the addition of a star network and an \Erdos random graph. Our theoretical bounds on the influence are particularly tight on this class of random graphs.

\section{Application to bond percolation}\label{sec:percolation}

Bounding the influence of reachable sets in random graphs allows to derive nonasymptotic bounds on a celebrated quantity in bond percolation which is the size of the giant component, as well as the distribution of the size of connected components of undirected random graphs. We first recall the inhomogeneous bond percolation model.

\begin{model}[\textbf{Bond percolation}]
A bond percolation graph is an \emph{undirected} random graph $\RandomG = \RandomG(n,A)$ of size $n$ with \emph{independent} edge  variables $\{A_{ij} : i < j\}$.
\end{model}

We recall that, according to \Lemma{lem:undirectedLPC}, $\RandomG$ is a random graph with LPC. For $k\geq 1$, we denote by $\mathcal{V}_k\subset \OneToN$ the $k^{th}$-largest connected component of $\RandomG$. We also introduce $C_k(G)= \card{\mathcal{V}_k}$, the size of the connected component with $k^{th}$ greatest cardinality and $N(m)$ the number of connected components of $\RandomG$ of cardinality greater than or equal to $m$.

\subsection{Size and existence of the giant component}
Let $a>0$. The key observation is that if each node of $\RandomG$ is an influencer with independent probability $1 - \exp(-a)$, we can relate the total size of infection and the size of the connected components of $\RandomG$ in the following way:
\begin{equation*}
\Exp{\card{R(\mathcal{I},A)}~|~A} = \sum_k {C_k(\RandomG)\Prob{\mathcal{V}_k \cap \mathcal{I} \neq \emptyset~|~A}} = \sum_k {C_k(\RandomG)\left(1-e^{-a C_k(\RandomG)}\right)}
\end{equation*}
Therefore, we have
\begin{equation}\label{eq:linkperco}
\ExpBigP{~\sum_k {C_k(\RandomG)\left(1-e^{-a C_k(\RandomG)}\right)}~} = \Exp{\sigma(\mathcal{I})} \leq \gamma(\HazSpec, a) n
\end{equation}
\medskip

The next argument leads to nonasymptotic bounds on the size of the giant component for the inhomogeneous bond percolation model:

\begin{theorem}[\textbf{Size of the giant component}]
\label{th:sizecomponent}
Let $\RandomG(n,A)$ be an undirected random graph of size $n$ with independent edge  variables $\{A_{ij} : i < j\}$. Let $a>0$, and $\HazSpec$ the Hazard radius of $\RandomG$. The probability distribution of the size of its largest connected component $C_1(\RandomG)$ verifies:
\begin{equation}
\label{eqn:sizecomp}
\Exp{C_1(\RandomG)(1-e^{-a (C_1(\RandomG)-1)})} \leq n \left(1-e^{-\HazSpec\gamma (\HazSpec,a)} \right)
\end{equation}
where $\gamma(\rho,a)$ is defined as in \Definition{def:gamma}.
\end{theorem}

\begin{corollary}\label{cor:simpleSize}
Under the same conditions as in \Theorem{th:sizecomponent}, we have:
\[
\displaystyle
\Exp{C_1(\RandomG)} \leq
\left\{
\begin{array}{ll}
 \displaystyle\frac{1}{2} + \sqrt{\frac{1}{4} +\frac{n \HazSpec}{1-\HazSpec}}~,  & \text{if } \HazSpec < 1 - \kappa n^{-1/3} \\
 &\\
  \displaystyle \gamma_0(\HazSpec) n + \frac{n^{2/3}}{\sqrt{\kappa}}~, &\text{if }  |\HazSpec - 1| \leq \kappa n^{-1/3}\\
  &\\
  \displaystyle \gamma_0(\HazSpec) n + c_n\sqrt{n} + 2~, & \text{if } \HazSpec > 1 + \kappa n^{-1/3}
\end{array}
\right.
\]
where $c_n = \displaystyle  \frac{2}{\sqrt{e}} \sqrt{\frac{(1-\gamma_0(\HazSpec))^2 \HazSpec}{1-\HazSpec+\gamma_0(\HazSpec)\HazSpec}}$ and $\kappa = \displaystyle \left(\frac{2e}{27}\right)^{2/3}$.
\end{corollary}

\begin{remark}[\textbf{Homogeneous percolation}]
Let $\mathcal{G} = (\mathcal{V}, \mathcal{E})$ be an undirected graph and $p\in[0,1]$. If $\Exp{A_{ij}} = p \one \{(i,j)\in\mathcal{E}\}$, then $\HazSpec = -\ln(1-p)\rho(\AdjMat)$ where $\AdjMat$ is the adjacency matrix of $\mathcal{G}$.
\end{remark}

Whereas the latter results hold for any $n \in \mathbb{N}$, classical results in percolation theory study the asymptotic behavior of sequences of graphs when $n \rightarrow \infty$. Up to our knowledge,  the best result in the inhomogeneous bond percolation model (see \cite{bollobas2007phase}, Corollary 3.2 of section 5) states that: under a given subcriticality condition, $C_1(\RandomG_{n})=o(n)$ \emph{asymptotically almost surely} (a.a.s.). Combining our previous theorem and Markov's inequality, we are in position of obtaining a significant improvement on the previous result.

\begin{corollary}\label{cor:limsup}
Denote by $G_n = \RandomG(n,A_n)$ a sequence of undirected random graphs and $\HazSpec$ the sequence of spectral radiuses of the corresponding Hazard matrices. Let $\omega_n$ be a sequence such that $\lim_{n \rightarrow \infty} \omega_n = +\infty$. We have:
\[
\displaystyle
C_1(\RandomG_{n}) =
\left\{
\begin{array}{ll}
 \displaystyle o(n^{1/2} \omega_n) \text{ a.a.s.}~,  & \text{if } \limsup_{n \rightarrow \infty} \HazSpec < 1 \\
 &\\
  \displaystyle o(n^{\max(1 - \beta,2/3)}\omega_n) \text{ a.a.s.}~, &\text{if } \HazSpec - 1 = O(n^{-\beta})
\end{array}
\right.
\]
\end{corollary}

Moreover, \Corollary{cor:simpleSize} also implies a result of non-existence of the giant component, in the sense that no component has a size proportional to the size of the network.
\begin{corollary}[\textbf{Existence of a giant component}]
Denote by $G_n = \RandomG(n,A_n)$ a sequence of undirected random graphs and $\HazSpec$ the sequence of spectral radiuses of the corresponding Hazard matrices. If $\limsup_{n\rightarrow \infty} \HazSpec \leq 1$, then
\begin{equation}
\Exp{C_1(\RandomG_n)} = o(n),
\end{equation}
and there is no giant component in $\RandomG_n$ a.a.s..
\end{corollary}
\begin{proof}
Combining the second equation of \Corollary{cor:simpleSize} (valid for all $\HazSpec \geq 0$) and \Lemma{lem:gammaSlope}, we obtain $\Exp{C_1(\RandomG_n)} = o(n)$. Markov's inequality implies the result in probability.
\end{proof}

\subsection{Number of components of cardinality larger than m}
The following result focuses on discovering the expectation of $N(m)$, the number of connected components of $\RandomG$ having cardinality greater than or equal to $m$. A straightforward observation yields $\Exp{N(m)} \leq \frac{n}{m}$. For critical and subcritical random graphs, we are able to show that the expectation of $N(m)$ is in fact decreasing much faster with respect to $m$.

\begin{theorem}
\label{th:numcomponents}
Let $\RandomG(n,A)$ be an undirected random graph of size $n$ with independent edge  variables $\{A_{ij} : i < j\}$. Let $m>0$. The expected number of connected components $N(m)$ of cardinality greater than or equal to $m$ is upper bounded by:
\begin{equation}
\Exp{N(m)} \leq \frac{n}{m} \min_{a>0} \left\{\frac{1-e^{-\HazSpec\gamma (\HazSpec,a)}}{1-e^{a(m-1)}}\right\}
\end{equation}
where $\gamma(\rho,a)$ is defined as in \Definition{def:gamma}.
\end{theorem}

\begin{corollary}\label{cor:simplenumcomponents}
Under the same conditions as in \Theorem{th:numcomponents}, we have:
\[
\displaystyle
\Exp{N(m)}  \leq
\left\{
\begin{array}{ll}
 \displaystyle \frac{n}{m(m-1)}\frac{\HazSpec}{1-\HazSpec}~,  & \text{if } \HazSpec < 1 - \kappa_1 m^{-1/2} \\
 &\\
  \displaystyle \frac{n}{m^{3/2}}\frac{1}{\kappa_1}~, &\text{if }  |\HazSpec - 1| \leq \kappa_1 m^{-1/2}\\
  &\\
  \displaystyle \frac{n}{m} \left(\gamma_0(\HazSpec) + \frac{c_n'}{\sqrt{m-1}} + \frac{c_n}{m-1} \right)~, & \text{if } \HazSpec > 1 + \kappa_1 m^{-1/2}
\end{array}
\right.
\]
where $c_n = \frac{(1-\gamma_0(\HazSpec))^2 \HazSpec}{1-\HazSpec+\gamma_0(\HazSpec)\HazSpec}$, $c_n' = \sqrt{\gamma_0(\HazSpec)c_n}$ and $\kappa_1 = \sqrt{\frac{\eta}{8}}\left(\sqrt{1+\frac{8}{2\eta - 1}} - 1\right)$ where $\eta$ is the strictly positive solution of $e^\eta = 2\eta + 1$ ($\kappa_1 \approx 0.32$).
\end{corollary}

\section{Application to site percolation}\label{sec:sitePerco}
In this section, we show that the results of the previous section can further be applied to site percolation.

\begin{model}[\textbf{Site percolation}]
A site percolation graph consists in removing the nodes of an undirected graph $\mathcal{G} = (\mathcal{V}, \mathcal{E})$ independently and with probability $1-p_i$. 
\end{model}

Although the resulting undirected random graph $\RandomG_{SP}$ is not LPC, the size of connected components in $\RandomG_{SP}$ can be bounded by the size of reachability sets of a random graph with LPC. Let $A_{ij} = X_j\one\{(i,j)\in\mathcal{E}\}$, where $X_i$ are the Bernouilli random variables indicating the presence or absence of a node $i$ in $\RandomG_{SP}$, and $\RandomG_{SP}' = \RandomG(n,A)$.
\begin{proposition}\label{prop:sitePerco}
$\RandomG_{SP}'$ is a random graph with LPC. Furthermore, if $R(\mathcal{I},A)$ is the reachable set of $\mathcal{I}\subset\OneToN$ in $\RandomG_{SP}'$, then
\begin{equation}
\card{R(\mathcal{I},A)} \geq \sum_k {C_k(\RandomG_{SP})\one\{\mathcal{V}_k \cap \mathcal{I} \neq \emptyset\}} \as,
\end{equation}
where $\mathcal{V}_k$ is the $k^{th}$-largest connected component of $\RandomG_{SP}$ and $C_k(\RandomG_{SP}) = \card{\mathcal{V}_k}$.
\end{proposition}
\begin{proof}
Since the $X_i$ are independent, for all $i,j,i',j'$, $A_{ij}$ and $A_{i'j'}$ are positively correlated if $j = j'$ and independent otherwise, which proves the LPC assumption. In order to prove the inequality, it suffices to see that, if there is an influencer in a connected component $\mathcal{V}_k$ of $\RandomG_{SP}$ (\ie $\mathcal{V}_k \cap \mathcal{I} \neq \emptyset$), then $\mathcal{V}_k$ is in reachable set $R(\mathcal{I},A)$.
\end{proof}

Since the inequality in \Proposition{prop:sitePerco} is the same starting point as the results derived for bond percolation, all the results of \Sec{sec:percolation} also apply to site percolation with the following Hazard radius:
\begin{equation}
\HazSpec = \rho\left(-\frac{\ln(1-p_i) + \ln(1-p_j)}{2}\one\{(i,j)\in\mathcal{E}\}\right).
\end{equation}

\section{Application to epidemiology}\label{sec:epidemiology}
In epidemiology, several models for the propagation of a disease in a population have been developped (\cite{Newman:2010:NI, kermack1932contributions, prakash2012threshold}), ranging from simple (\eg SI, SIS, SIR \cite{Newman:2010:NI, kermack1932contributions}) to more complex (\eg SIRS, SEIR, SEIV \cite{prakash2012threshold}) diffusion mechanisms. We here focus on the standard Susceptible-Infected-Removed (SIR) model, and show that its long-term behavior is a particular case of LPC random graphs.

\begin{model}[\textbf{Susceptible-Infected-Removed \cite{kermack1932contributions}}]
Let $\mathcal{G} = (\mathcal{V}, \mathcal{E})$ be an undirected network of $n$ nodes, $\AdjMat = \big(\one\{(i,j)\in \mathcal{E}\}\big)_{ij}$ its adjacency matrix, and $\delta, \beta > 0$. A Susceptible-Infected-Removed epidemic $SIR(\mathcal{G},\delta, \beta)$ is a stochastic process $(S_t, I_t,$ $R_t)_{t\geq 0}$, where $S_t$, $I_t$ and $R_t$ 
encode the state of each node of the network $\mathcal{G}$ during the epidemic by partitioning the nodes into three sets, depending on their infectious state: \emph{susceptible}, \emph{infected} or \emph{removed}.
Each edge of the graph transmits the disease at rate $\beta$, and each infected node recovers at rate $\delta$. More formally, let $t^I_i = \{t\geq 0 : i\in I_t\}$ be the time when node $i$ becomes infected. Then $\forall i\in I_0$, $t^I_i = 0$ and
\begin{equation}
\forall i\notin I_0, t^I_i = \min_{\{j\in\OneToN : T_{ji} < D_j\}} (t^I_j + T_{ji}),
\end{equation}
where $T_{ji}$ and $D_i$ are independent exponential random variables of expected value $1/\beta$ and $1/\delta$, respectively. Then, the recovery time of each node is $t^R_i = t^I_i + D_i$, and the sets $S_t$, $I_t$ and $R_t$ are given by $S_t = \{i\in\OneToN : t < t^I_i\}$, $I_t = \{i\in\OneToN : t^I_i \leq t < t^R_i\}$ and $R_t = \{i\in\OneToN : t^R_i \leq t\}$.
\end{model}

\subsection{Subcritical behavior in the standard SIR model}
We show here that \Theorem{th:mainResult} (through \Corollary{cor:simpleBounds}) further improves results on the SIR model in epidemiology. In order to determine the long-term behavior of the epidemic, the following theorem shows that the set $\lim_{t\rightarrow+\infty} R_t$ of recovered nodes is the reachable set of a random graph with LPC.

\begin{proposition}
Let $S_t$, $I_t$ and $R_t$ be (respectively) the sets of susceptible, infected and removed individuals at time $t\geq 0$ of an $SIR(\mathcal{G}, \delta, \beta)$ epidemic with transmission times $T_{ij}$ and recovery times $D_i$. Then, the random graph $\RandomG_{SIR} = \RandomG(n,A)$ with adjacency matrix $A_{ij} = \one\{\{i,j\}\in\mathcal{E} \mbox{ and } T_{ij} < D_i\}$ is a random graph with LPC and, if $R(I_0,A)$ is the reachable set of $I_0$ in $\RandomG_{SIR}$, then $\lim_{t\rightarrow +\infty} R_t = R(I_0,A)$.
\end{proposition}
\begin{proof}
$\RandomG_{SIR}$ is a random graph with LPC since only outgoing edges of a node are correlated together, and this correlation is positive due to the fact that $\Exp{A_{ij}|A_{-ij} = a} = \Prob{T_{ij} < D_i ~|~ \max_{k\in\mathcal{N}_{ij}^1(a)} T_{ik} < D_i \leq \min_{k\in\mathcal{N}_{ij}^0(a)} T_{ik}}$, where $\mathcal{N}_{ij}^b(a) = \{k\neq j \mbox{ such that } \{k,i\}\in\mathcal{E} \mbox{ and } a_{ik} = b\}$ for $b\in\{0,1\}$, which is non-decreasing \wrt $a$. Finally, $\lim_{t\rightarrow +\infty} R_t = \{i\in\OneToN : t^I_i < +\infty\}$, and a node $i$ is in $\lim_{t\rightarrow +\infty} R_t$ if and only if $i\in I_0$ or $i$ has an infected neighbor $j$ that transmitted the disease, \ie such that $j\in R(I_0,A)$, $\{j,i\}\in\mathcal{E}$ and $T_{ji} < D_j$, or equivalently $A_{ji} = 1$.
\end{proof}

In the sequel, we will refer to the number of infected nodes through the epidemic as $\sigma(\mathcal{I})$. The  Hazard matrix $\mathcal{H}$ of $\RandomG_{SIR}$ is  given by  $\ln(1+\frac{\beta}{\delta}) \cdot \AdjMat$ and hence $\HazSpec = \ln(1+\frac{\beta}{\delta}) \cdot \rho(\AdjMat)$. A direct application of
Corollary \ref{cor:simpleBounds} leads to the following result.

\begin{corollary}\label{cor:epiThSIR}
We consider an $SIR(\mathcal{G},\delta,\beta)$ epidemic with $\delta>0$. We denote by $\AdjMat$ the symmetric adjacency matrix of $\mathcal{G}$, by $\HazSpec = \ln(1+\frac{\beta}{\delta}) \cdot \rho(\AdjMat)$ its Hazard radius, and by $\mathcal{I}$ the initial set of influencers of size $n_0$.  If  $\frac{\beta}{\delta} < \exp(\frac{1}{\rho(\AdjMat)}) - 1$, then we have
\begin{equation}
\sigma(\mathcal{I}) \le n_0+\sqrt{\frac{\HazSpec}{1-\HazSpec}} \sqrt{n_0(n-n_0)} ~.
\end{equation}
\end{corollary}

It was recently shown by Draief, Ganesh and Massouli\'e (\cite{draief2008}) that, in the case of undirected networks, and if $\beta \rho(\AdjMat)<\delta$, we have the following bound on the influence for a fixed set of influencer nodes:
\begin{equation}
\label{eqn:Massoulie}
\sigma(\mathcal{I}) \leq \frac{\sqrt{n n_0}}{1-\frac{\beta}{\delta} \rho(\AdjMat)}~.
\end{equation}
As we will show now, Corollary \ref{cor:epiThSIR} improves the result of \cite{draief2008, prakash2012threshold} in two directions: weaker condition and tighter constants in the upper bound. Indeed, when $\rho(\AdjMat) \gg 1$, $\frac{1}{\rho(\AdjMat)}$ is a good approximation of $\exp(\frac{1}{\rho(\AdjMat)}) - 1$. However, the two quantities may differ substantially on very sparse networks, for which $\rho(\AdjMat)$ is close to $1$. For example, for an n-cycle graph, we have $\AdjMat_{ij} = \one\{j = i \pm 1 \mod n\}$ where mod is the modulo operator, which leads to $\rho(\AdjMat) = 2$ and $\exp(\frac{1}{\rho(\AdjMat)}) - 1 \approx 0.65 > 0.5$.
Now, as far as the comparison between the two rates is concerned, we offer the following lemma which assesses the tightness of Corollary \ref{cor:epiThSIR} with respect to the upper bound in \Eq{eqn:Massoulie}.

\begin{lemma}\label{lem:lemmaMassoulie}
We use the same notations as in Corollary \ref{cor:epiThSIR}. If
$\beta \rho(\AdjMat) < \delta$, then we have:
\begin{equation}
n_0+\sqrt{\frac{\HazSpec}{1-\HazSpec}} \sqrt{n_0(n-n_0)} \leq \frac{\sqrt{n n_0}}{1-\frac{\beta}{\delta} \rho(\AdjMat)}~.
\end{equation}
\end{lemma}

\begin{proof}
First, $\HazSpec = \ln(1+\frac{\beta}{\delta})\rho(\AdjMat) \leq \frac{\beta}{\delta}\rho(\AdjMat)$. Then, we introduce the function $$f: r \rightarrow n_0+\sqrt{\frac{r}{1-r}} \sqrt{n_0(n-n_0)} - \frac{\sqrt{n n_0}}{1- r}~.$$ A simple analysis shows that $\max_{r\in[0,1]}f(r) = n_0\left(1 - \frac{3}{4}\sqrt{\frac{n}{n_0}} - \frac{1}{4}\sqrt{\frac{n_0}{n}}\right) \leq 0$ when $n_0 \leq n$, which proves the lemma.
\end{proof}

Moreover, these new bounds capture with increased accuracy the behavior of the influence in extreme cases. In the limit $\beta \rightarrow 0$, the difference between the two bounds is significant, because \Theorem{th:mainResult} yields $\sigma(\mathcal{I}) \rightarrow n_0$ whereas \Eq{eqn:Massoulie} only ensures $\sigma(\mathcal{I}) \leq \sqrt{n n_0}$. When $n=n_0$, \Theorem{th:mainResult} also ensures that $\sigma(\mathcal{I}) = n_0$ whereas \Eq{eqn:Massoulie} yields $ \sigma(\mathcal{I}) \leq \frac{n_0}{1-\frac{\beta}{\delta} \rho(\AdjMat)} $. Secondly, \Theorem{th:mainResult} also describes the explosive behavior in the SIR model and leads to bounds in the case where $\beta \rho(\AdjMat) \geq \delta$, as we will see below.

\subsection{Behavior near the epidemic threshold}
The regime around $1$ of $\ln(1+\frac{\beta}{\delta}) \cdot \rho(\AdjMat)$ can also be derived from the generic results of \Sec{sec:bounds}.
\begin{corollary}[\textbf{Critical behavior of SIR}]
We use the same notations as in Corollary \ref{cor:epiThSIR}. If $|\HazSpec - 1| \leq \left(\frac{n_0}{4(n - n_0)}\right)^{1/3}$, then we have
\begin{equation}
\sigma(\mathcal{I}) \leq n_0 + 2^{4/3} n_0^{1/3} (n - n_0)^{2/3} ~.
\end{equation}
\end{corollary}
\begin{proof}
This is also a direct application of \Corollary{cor:simpleBounds} to $\RandomG_{SIR}$.
\end{proof}

More specifically, the behavior when $\beta \rho(\AdjMat) = \delta$ depends on the rate at which the spectral radius of the adjacency matrix diverges \wrt $n$.

\begin{corollary}
We use the same notations as in \Corollary{cor:epiThSIR}. Assume $n_0=O(1)$ and $\rho(\AdjMat) = O(n^\alpha)$ for $\alpha \geq 0$. If $\beta \rho(\AdjMat) = \delta$, then we have
\begin{equation}
\sigma(\mathcal{I}) = O\left(n^{\min\{\frac{1+\alpha}{2}, \frac{2}{3}\}}\right) ~.
\end{equation}
\end{corollary}
\begin{proof}
If the graph is empty, then $\sigma(\mathcal{I}) = 0$. Otherwise, $\rho(\AdjMat) \geq 1$ and $\HazSpec = \ln(1+\frac{1}{\rho(\AdjMat)})\rho(\AdjMat) \leq 1 - \frac{1-\ln{2}}{\rho(\AdjMat)}$, the critical bound of \Corollary{cor:simpleBounds} implies that $\sigma(\mathcal{I}) = O(n^{2/3})$, while the subcritical  bound implies that $\sigma(\mathcal{I}) = O(n^{\frac{1+\alpha}{2}})$.
\end{proof}

The behavior in $O(n^{2/3})$ of the size of the epidemic in the critical regime was already known for the more simple N-intertwinned SIR model \cite{PhysRevE.69.050901} (in which the three populations are assumed to be mixed uniformly). However, this result is, up to our knowledge, the first to prove such a behavior in the more general case of epidemics on networks. Finally, note that \Proposition{prop:tightness} implies that the behavior in $O(n^{2/3})$ is tight, in the sense that some networks do behave accordingly in the critical regime.

\subsection{Generic incubation period}
\Theorem{th:mainResult} applies to more general cases than the classical homogeneous SIR model, and allows infection and recovery rates to vary across individuals. Also, our model allows for incubation times which display a non-exponential behavior, and thus is more adapted to realistic scenarios. Indeed, incubation periods for different individuals generally follow a log-normal distribution \cite{nelson2007epidemiology}, which indicates that SIR with a log-normal recovery rate of removal might be well-suited to model real-world infections.

For each node $i$, let the incubation time $D_i$ (\ie the time for an infected node to recover) be a random variable drawn according to a certain probability distribution $P_D$. In such a case, the Hazard radius is
\begin{equation}
\HazSpec = -\rho(\AdjMat)\ln\left(\Exp{e^{-\beta D}}\right),
\end{equation}
and a sufficient condition for subcriticality is $\beta \rho(\mathcal{A}) \Exp{D} < 1$, where $\Exp{D} = \int x P_D(x)dx$.

\begin{corollary}[\textbf{Generic incubation period}]\label{cor:genericIncubationTime}
We consider a graph of contaminated nodes obtained after the realization of an SIR contagion process with incubation times drawn according to the probability distribution $P_D$. We denote by $\AdjMat$ its symmetric adjacency matrix, by $\mathcal{I}$ its initial set of influencers of size $n_0=O(1)$, and $\Exp{D} = \int x P_D(x)dx$.

If $\beta \rho(\mathcal{A}) \Exp{D} < 1$, then we have
\begin{equation}
\sigma(\mathcal{I}) = O(\sqrt{n}) ~.
\end{equation}
\end{corollary}
\begin{proof}
Since, for $\RandomG_{SIR} = \RandomG(n,A)$, $\Exp{A_{ij}} = \Prob{T_{ij} < D_i} = 1 - \Exp{e^{-\beta D}}$, a direct application of \Corollary{cor:simpleBounds} to $\RandomG_{SIR}$ returns that the epidemic is subcritical if $\HazSpec = -\rho(\AdjMat)\ln\left(\Exp{e^{-\beta D}}\right) < 1$. Jensen's inequality on $\Exp{e^{-\beta D}}$ leads to the desired result.
\end{proof}

In the log-normal case, \Corollary{cor:genericIncubationTime} gives the following bound on the epidemic threshold:
\begin{corollary}[\textbf{Log-normal incubation period}]
We consider a graph of contaminated nodes obtained after the realization of an SIR contagion process with incubation times drawn according to a log-normal distribution of parameters $\mu_D$ and $\sigma_D$. We denote by $\AdjMat$ its symmetric adjacency matrix, by $\mathcal{I}$ its initial set of influencers of size $n_0=O(1)$.
If ${\displaystyle \mu_D + \frac{\sigma_D^2}{2} < -\ln\left(\beta \rho(\mathcal{A})\right)}$, then we have
\begin{equation}
\sigma(\mathcal{I}) = O(\sqrt{n}) ~.
\end{equation}
\end{corollary}

\section{Application to Information Cascades}\label{sec:IC}
In information propagation theory, information cascades have emerged as a relevant model for viral diffusion of ideas and opinions \cite{Kempe:2003:MSI:956750.956769, chen2009efficient, rodriguez2012influence, DBLP:conf/icml/Gomez-RodriguezBS11}. They are of two types:

\begin{model}[\textbf{Discrete-Time Information Cascades [$DTIC(\mathcal{P})$, \cite{chen2009efficient, Kempe:2003:MSI:956750.956769}]}]
At time $t=0$, only a set $\mathcal{I}$ of influencers is infected. Given a matrix $\mathcal{P}=(p_{ij})_{ij} \in [0,1]^{n\times n}$, each node $i$ that receives the contagion at time $t$ may transmit it at time $t+1$ along its outgoing edge $(i, j)\in\mathcal{E}$ with probability $p_{ij}$. Node $i$ cannot make any attempt to infect its neighbors in subsequent rounds. The process terminates when no more infections are possible.
\end{model}

\begin{model}[\textbf{Continuous-Time Information Cascades [$CTIC(\mathcal{F},T)$, \cite{rodriguez2012influence, DBLP:conf/icml/Gomez-RodriguezBS11}]}]
At time $t=0$, only a set $\mathcal{I}$ of influencers is infected. Given a matrix $\mathcal{F}=(f_{ij})_{ij}$ of non-negative integrable functions, each node $i$ that receives the contagion at time $t$ may transmit it at time $s>t$  along its outgoing edge $(i, j)\in\mathcal{E}$ with stochastic rate of occurrence $f_{ij}(s-t)$. The process terminates at a given deterministic time $T>0$. This model is much richer than the DTIC model, but we will focus here on its behavior when $T = \infty$.
\end{model}

Let $I_t\subset\OneToN$ be the set of infected nodes at time $t$. In $DTIC(\mathcal{P})$ and $CTIC(\mathcal{F},\infty)$, $I_t$ is non-decreasing \wrt $t$ and reaches a limit set $I_{\infty} = \lim_{t\rightarrow+\infty}I_t$. Due to the independence of transmission events along the edges of the graph, $I_{\infty}$ is the reachable set of $\mathcal{I}$ in a random graph $\RandomG_{IC} = \RandomG(n, A)$ with independent edge variables $A_{ij}$. Hence $\RandomG_{IC}$ is a random graph with LPC and the results of \Sec{sec:bounds} are applicable.

\begin{proposition}
Let $\mathcal{I}$ be a set of influencers, $\mathcal{P}=(p_{ij})_{ij} \in [0,1]^{n\times n}$ a matrix of transmission probabilities and $\mathcal{F}=(f_{ij})_{ij}$ a matrix of non-negative integrable functions. Then, under $DTIC(\mathcal{P})$ and $CTIC(\mathcal{F},\infty)$, the set of infected nodes at the end of the diffusion process is the reachable set $R(\mathcal{I},A)$ of $\mathcal{I}$ in a random graph with LPC, and \Theorems{th:mainResult}, \ref{th:uniformResult} and \ref{th:randomResult} are applicable with the following Hazard matrix:
\begin{equation}
\HazMat_{ij} = \left\{
\begin{array}{ll}
-\ln(1-p_{ij}) &\mbox{for } DTIC(\mathcal{P})\\
\int_0^\infty f_{ij}(t)dt &\mbox{for } CTIC(\mathcal{F},\infty)\\
\end{array}\right..
\end{equation}
\end{proposition}
\begin{proof}
Since transmission events are independent, we can, prior to the epidemic, draw, respectively, the transmission along each edge $(i,j)\in\OneToN^2$ $T_{ij}\sim\mathcal{B}(p_{ij})$ for $DTIC(\mathcal{P})$, and time to transmit along each edge $\tau_{ij}\sim p_{ij}(t) = f_{ij}(t)e^{-\int_0^t f_{ij}(u)du}$ for $CTIC(\mathcal{F},T)$. Then, a node $i$ belongs to $\mathcal{I}_\infty$ if and only if there is a path between $\mathcal{I}$ and $i$ such that each of its edges transmitted the information. Hence, $\mathcal{I}_\infty$ is the reachable set of $\mathcal{I}$ in the random graph $\RandomG(n,A)$ \st $A_{ij} = T_{ij}$ for $DTIC(\mathcal{P})$, and $A_{ij} = \one\{\tau_{ij} < +\infty\}$ for $CTIC(\mathcal{F},T)$. These are independent Bernoulli random variables of parameter $p_{ij}$ for $DTIC(\mathcal{P})$, and $1 - \exp(-\int_0^\infty f_{ij}(t)dt)$ for $CTIC(\mathcal{F},T)$, which implies that $\RandomG$ is a random graph with LPC and the above mentioned Hazard matrices.
\end{proof}

\section{Conclusion}
In this paper, we established new bounds on the influence in random graphs, and applied our results to three quantities of major importance in their respective fields: the size of the giant component in percolation, the number of infected nodes in epidemiology and the influence of information cascades. These bounds are a strong indication that the \emph{Hazard radius} plays an important role in the dynamics of diffusion processes in random graphs, and lead to several open questions. 
First, one may wonder if the LPC property is a necessary condition for the bounds to hold. For example, relaxing the local correlation and allowing positive correlation on larger neighborhoods may still provide random graphs in which criticality is controlled by the Hazard radius. Second, an important class of diffusion models, based on randomized versions of the \emph{Linear Threshold} model, is so far absent of this analysis, and being able to describe such models by a well chosen LPC random graph may lead to new and valuable results.
Finally, the Hazard radius may drive the behavior of other diffusion-related quantities in random graphs, such as the volume of neighborhoods of fixed size. Such results would prove critical for understanding the temporal dynamics of diffusion processes in networks.

\appendix

\section{Behavior of the Hazard function}\label{appendix:gamma}
When $\rho \geq 0$ and $a > 0$, $\gamma - 1 + \exp\left(-\rho\gamma - a\right) = 0$ always has a solution in $[0,1]$ and $\gamma(\rho, a)$ is well defined. $\gamma$ and $\gamma_0$ are non-decreasing \wrt $\rho$, $\lim_{\rho \rightarrow +\infty} \gamma(\rho,a) = \lim_{\rho \rightarrow +\infty} \gamma_0(\rho) = 1$ and, for $\rho \leq 1$, $\gamma_0(\rho) = 0$.

Moreover, we have the following upper bounds for $\gamma(\rho,a)$, that we will use to determine the subcritical, critical and supercritical behavior of the influence.

\begin{lemma}\label{lem:gammaBounds}
$\forall \rho \neq 1$ and $a > 0$,
\begin{equation}\label{eq:gammaApprox1}
\gamma(\rho,a) \leq \gamma_0(\rho) + \frac{a(1 - \gamma_0(\rho))}{1 - \rho(1-\gamma_0(\rho))},
\end{equation}
and $\forall \rho > 0$ and $a > 0$,
\begin{equation}
\gamma(\rho,a) \leq \gamma_0(\rho) + \sqrt{2a} \min\left\{1, \sqrt{\frac{1}{\rho}}\right\}.
\end{equation}
\end{lemma}

\Eq{eq:gammaApprox1} is particularly tight, except when $\rho\approx 1$ (\ie the critical case). In order to derive upper bounds in the critical case, we will thus use the second upper bound.

\begin{proof}
By definition of $\gamma(\rho, a)$,
\begin{equation}
\begin{array}{ll}
\gamma(\rho, a) &= 1 - \exp\left(-\rho\gamma_0(\rho) - \rho(\gamma(\rho, a) - \gamma_0(\rho)) - a\right)\\
                &\leq 1 - (1 - \gamma_0(\rho))\left(1 - \rho(\gamma(\rho, a) - \gamma_0(\rho)) - a\right)\\
								&= \gamma_0(\rho) + (1 - \gamma_0(\rho))\left(\rho(\gamma(\rho, a) - \gamma_0(\rho)) + a\right),
\end{array}
\end{equation}
hence
\begin{equation}
\gamma(\rho, a) \leq \gamma_0(\rho) + \frac{a(1 - \gamma_0(\rho))}{1 - \rho(1-\gamma_0(\rho))}.
\end{equation}

For the second inequality, first observe that $\gamma(\rho, a) \geq 1 - \frac{1}{\rho}$ since $\gamma(\rho, a) = 1 - \exp(-\rho\gamma(\rho,a) - a) \geq 1 - \frac{1}{1 + \rho\gamma(\rho,a)} = \frac{\rho\gamma(\rho,a)}{1 + \rho\gamma(\rho,a)}$ which leads to $\rho\gamma(\rho,a) \geq \rho-1$. The second inequality follows from an approximation of the derivative of $\gamma(\rho,a)$ \wrt $a$:
\begin{equation}
\frac{\partial\gamma(\rho,a)}{\partial a} = \frac{1-\gamma(\rho,a)}{1 - \rho(1-\gamma(\rho,a))} \leq \frac{1}{1 - \rho(1-\gamma(\rho,a))}.
\end{equation}
Multiplying the two terms by $1 - \rho(1-\gamma(\rho,a)) > 0$ and integrating between $0$ and $a$, we get
\begin{equation}
(1-\rho)(\gamma(\rho,a) - \gamma_0(\rho)) + \frac{\rho}{2}(\gamma(\rho,a)^2 - \gamma_0(\rho)^2) \leq a
\end{equation}
which leads to
\begin{equation}
\gamma(\rho,a) \leq 1-\frac{1}{\rho} + \sqrt{(\gamma_0(\rho) - 1 + \frac{1}{\rho})^2 + \frac{2a}{\rho}} \leq \gamma_0(\rho) + \sqrt{\frac{2a}{\rho}}.
\end{equation}
using that, $\forall a,b\geq 0$, $\sqrt{a+b}\leq\sqrt{a}+\sqrt{b}$. Finally, noting that $\gamma(\rho,a)$ is non-decreasing, we get that $\forall \rho \leq 1, \gamma(\rho,a) \leq \gamma(1, a) \leq \sqrt{2a}$, and $\forall \rho \geq 1, \sqrt{\frac{2a}{\rho}} \leq \sqrt{2a}$.
\end{proof}

We will also use the follwing bound on $\gamma_0(\rho)$:
\begin{lemma}\label{lem:gammaSlope}
$\forall \rho \geq 1$, $\gamma_0(\rho) \leq 2(\rho - 1)$.
\end{lemma}
\begin{proof}
A simple calculation holds that $\gamma_0$ is concave on $(1,+\infty)$. Thus, it implies that, $\forall\rho > 1$,
\begin{equation}
\gamma_0(\rho) \leq \gamma_0'(1^+)(\rho-1).
\end{equation}
Finally, $\forall \epsilon > 0$,
\begin{equation}
\gamma_0'(1+\epsilon) = \frac{\gamma_0(1+\epsilon)(1-\gamma_0(1+\epsilon))}{1 - (1+\epsilon)(1-\gamma_0(1+\epsilon))} = \frac{\gamma_0'(1^+)}{\gamma_0'(1^+)-1} + o(\epsilon),
\end{equation}
which leads to $\gamma_0'(1^+) \in \{0, 2\}$, and $\gamma_0'(1^+) = 0$ is impossible since $\gamma_0$ is concave on $(1,+\infty)$ and $\gamma_0(\rho) > 0$ for all $\rho > 1$. Hence, $\gamma_0'(1^+) = 2$ and $\gamma_0(\rho) \leq 2(\rho - 1)$ for $\rho \geq 1$.
\end{proof}

\section{Proofs of the upper bounds on influence}
In this section, we consider $\RandomG(n,A)$ a random graph with LPC and $R(\mathcal{I},A)$ the reachable set of a set of influencers $\mathcal{I}$ in $\RandomG$. We will also define $X_i = \one\{i\in R(\mathcal{I},A)\}$ the indicators of the reachable set.
First, note that all the bounds provided in \Sec{sec:bounds} are infinite when there exists an edge $(i,j)$ such that $\Exp{A_{ij}} = 1$ (since, in such a case, $\HazSpec = +\infty$). Hence, we will assume that $\forall (i,j)\in\OneToN^2, \Exp{A_{ij}} < 1$. In the following paragraphs, we will prove our results for random graphs having a \emph{strictly positive} measure, \ie such that every graph of $n$ nodes has a non-zero probability. When this assumption is not satisfied, the next two lemmas show that we can still derive the desired results by considering a sequence of such graphs converging to $\RandomG$.

\begin{definition}[\textbf{Perturbed random graph}]\label{def:perturbedGraph}
Let $\RandomG(n, A)$ be a random graph and $ \varepsilon > 0$. The \emph{$ \varepsilon$-perturbed version of $\RandomG$}, $\RandomG^ \varepsilon = (n, A^ \varepsilon)$, is a random graph such that $A^ \varepsilon_{ij} = A_{ij}(1 - X_{ij}) + Y_{ij}X_{ij}$ where $X_{ij}$ and $Y_{ij}$ are, respectively, i.i.d. Bernoulli random variables with parameter $ \varepsilon$ and $1/2$, and independent of $\RandomG$.
\end{definition}

These \emph{noisy} versions of $\RandomG$ have a strictly positive measure, while still verifying the LPC property.

\begin{lemma}\label{lem:lemmaPerturbed1}
Let $ \varepsilon > 0$, $\RandomG(n, A)$ a random graph and $\RandomG^ \varepsilon = (n, A^ \varepsilon)$ its $ \varepsilon$-perturbed version. Then $\RandomG^ \varepsilon$ has a \emph{strictly positive} measure and, if $\RandomG$ is a random graph with LPC, then so does $\RandomG^ \varepsilon$.
\end{lemma}
\begin{proof}
Let $X$ and $Y$ be the random matrices of \Definition{def:perturbedGraph}. $\forall a\in\{0,1\}^{n^2}$, $\Prob{A^ \varepsilon = a} \geq \Prob{X = \one, Y = a} = (\frac{ \varepsilon}{2})^{n^2} > 0$, where $\one$ is a vector of size $n$ filled with ones. Also, since $X_{ij}$ and $Y_{ij}$ are independent, then (H1) and (H2) of the LPC property are still verified for $\RandomG^ \varepsilon$ (see \Definition{def:LPC}).
\end{proof}

Furthermore, the influence and Hazard radius are continuous \wrt $ \varepsilon$, and thus our results on strictly positive measures can be generalized to any random graph with LPC.

\begin{lemma}\label{lem:lemmaPerturbed2}
Let $\mathcal{I}$ be a set of influencers, $ \varepsilon > 0$, $\RandomG(n, A)$ a random graph and $\RandomG^ \varepsilon = \RandomG(n, A^ \varepsilon)$ its $ \varepsilon$-perturbed version. Let also $\sigma(\mathcal{I})$ be the influence of $\mathcal{I}$ in $\RandomG$, $\sigma^ \varepsilon(\mathcal{I})$ the influence of $\mathcal{I}$ in $\RandomG^ \varepsilon$, $\HazSpec$ the Hazard radius of $\RandomG$ and $\rho^ \varepsilon_\HazMat$ the Hazard radius of $\RandomG^ \varepsilon$. Then the following results hold:
\begin{equation}
\lim_{ \varepsilon \rightarrow 0} \sigma^ \varepsilon(\mathcal{I}) = \sigma(\mathcal{I}),
\end{equation}
and
\begin{equation}
\lim_{ \varepsilon \rightarrow 0} \rho^ \varepsilon_\HazMat = \HazSpec.
\end{equation}
\end{lemma}
\begin{proof}
Let $X$ and $Y$ be the random matrices of \Definition{def:perturbedGraph}. $\forall a\in\{0,1\}^{n^2}$, $\Prob{A^ \varepsilon = a} = \Prob{X = \zero, A = a} + \Prob{X \neq \zero, A^ \varepsilon = a}$, where $\zero$ is the vector of size $n$ filled with zeros. Hence,
\begin{equation}
\begin{array}{ll}
|\Prob{A^ \varepsilon = a} - \Prob{A = a}| &= |\Prob{X \neq \zero, A^ \varepsilon = a} - \Prob{X \neq \zero}|\\
&\leq 2\Prob{X \neq \zero}\\
&\leq 2(1 - (1 -  \varepsilon)^{n^2})\\
&\rightarrow_{ \varepsilon \rightarrow 0} 0.
\end{array}
\end{equation}
Since $\{0,1\}^{n^2}$ is finite, $A^ \varepsilon$ converges to $A$ in law, and for any function $f: \{0,1\}^{n^2} \rightarrow \mathbb{R}$ we have:
\begin{equation}
\lim_{ \varepsilon \rightarrow 0} \Exp{f(A^ \varepsilon)} = \Exp{f(A)}.
\end{equation}
Selecting $f(A) = \card{R(\mathcal{I},A)} = \card{\mathcal{I}} + \sum_{i\notin\mathcal{I}} \left(1 - \prod_{q\in\mathcal{Q}_{\mathcal{I},i}}(1 - \prod_{(j,l)\in q}A_{jl})\right)$ (see \Definition{def:reachableSet}) implies that $\lim_{ \varepsilon \rightarrow 0} \sigma^ \varepsilon(\mathcal{I}) = \sigma(\mathcal{I})$. The second result comes from the continuity of the spectral radius and that
\begin{equation}
\HazMat^ \varepsilon_{ij} = -\ln\left(1 - (1- \varepsilon)\Exp{A_{ij}} - \frac{ \varepsilon}{2}\right) \rightarrow_{ \varepsilon \rightarrow 0} \HazMat_{ij}.
\end{equation}
\end{proof}

\subsection{Proofs of \Theorem{th:mainResult} and \Corollary{cor:simpleBounds}}

We develop here the full proofs for \Theorem{th:mainResult} and \Corollary{cor:simpleBounds} that apply to any set of influencers. Due to \Lemma{lem:lemmaPerturbed1} and \Lemma{lem:lemmaPerturbed2}, without loss of generality, we will restrict ourselves to random graphs $\RandomG$ that have a strictly positive measure. We will first need to prove two useful results: \Lemma{lem:poscorrelated}, that proves for $j \in \OneToN$ a positive correlation between the events 'node $i$ is not reachable from $\mathcal{I}$ through node $j$' and \Lemma{lem:mainLemma}, that bounds the probability that a given node is reachable from $\mathcal{I}$.

\begin{lemma}\label{lem:poscorrelated}
$\forall i \notin \mathcal{I}$, $\{1 - X_j A_{ji}\}_{j\in \OneToN}$ are positively correlated.
\end{lemma}

\begin{proof}
We will make use of a generalization of the FKG inequality due to Holley \cite{holley1974,Georgii99therandom}, that only requires the positive correlation of the edge presence variables $A_{ij}$ (hypothesis (H2) of the LPC property, see \Definition{def:LPC}):
\begin{lemma}[\textbf{FKG inequality (Theorem 4.11 of \cite{Georgii99therandom} adapted to our notations)}]\label{lem:FKG}
Let $\mathcal{L}$ be finite, $S$ a finite subset of $\mathbb{R}$, $\mu$ a strictly positive probability measure on $S^\mathcal{L}$, and $X\in S^\mathcal{L}$ a random variable with probability measure $\mu$. If $\mu$ is \emph{monotone}, \ie $\forall i\in \mathcal{L}$ and $a\in S$, $\xi \mapsto \ProbUnder{\mu}{X_i \geq a | X_{\mathcal{L}\setminus\{i\}} = \xi}$ is non-decreasing \wrt the natural partial order on $S^{\mathcal{L}\setminus\{i\}}$, then it also has \emph{positive correlations}: for any bounded non-decreasing functions $f$ and $g$ on $S^\mathcal{L}$
\begin{equation}
\ExpUnder{\mu}{f(X)g(X)} \geq \ExpUnder{\mu}{f(X)}\ExpUnder{\mu}{g(X)}.
\end{equation}
\end{lemma}
In our setting, $\mathcal{L} = \OneToN^2$, $S = \{0,1\}$, and $\mu$ is the probability measure of the adjacency matrix $A$. For a given set of influencers $\mathcal{I}$, the indicator values of the reachable set $X_i = \one\{i\in R(\mathcal{I},A)\}$ are deterministic functions of the random variables $A_{ij}$. Thus, let $f_{ij}(\{A_{i'j'}\}_{(i',j')}) = 1 - X_jA_{ji}$. In order to apply the FKG inequality, we first need to show that each $f_{ij}: \{0, 1\}^{n^2} \rightarrow \{0, 1\}$ is non-increasing with respect to the natural partial order on $\{0, 1\}^{n^2}$ (\ie $X \leq Y$ if $X_i \leq Y_i$ for all $i$). Let $u\in\{0, 1\}^{n^2}$ be a given state of the edges of the network. In order to prove the non-increasing behavior of $f_{ij}$, it is sufficient to show that $f_{ij}(u)$ is non-increasing with respect to every $u_{(i,j)}$.

But from \Definition{def:reachableSet}, it is obvious that $X_i(u) = 1 - \prod_{q\in\mathcal{Q}_i}(1 - \prod_{(j,l)\in q}u_{(j,l)})$ is non-decreasing with respect to every $u_{(i,j)}$. This implies that $f_{ij}(u) = 1 - X_j(u) u_{(j,i)}$ is non-increasing with respect to every $u_{(i,j)}$ and that $f_{ij}: \{0, 1\}^{n^2} \rightarrow \{0, 1\}$ is non-increasing with respect to the natural partial order on $\{0, 1\}^{n^2}$.

Finally, since the LPC property implies that the probability measure of $A$ is monotonic, we can apply the FKG inequality to $\{1 - X_j A_{ji}\}_{j\in \OneToN}$, and these random variables are positively correlated.
\end{proof}

The next lemma ensures that the variables $X_i$ satisfy an implicit inequation that will be the starting point of the proof of \Theorem{th:mainResult}.

\begin{lemma}\label{lem:mainLemma}

For any $\mathcal{I}$ such that $\card{\mathcal{I}}=n_0 < n$ and for any $i \notin \mathcal{I}$, the probability $\Exp{X_i}$  that node $i$ is reachable from $\mathcal{I}$ in $\RandomG$ verifies:

\begin{equation}
\Exp{X_i} \leq 1-\exp \bigg(- \sum_j \HazMat_{ji} \Exp{X_j} \bigg)
\end{equation}
\end{lemma}

\begin{proof}
We first note that a node $i\notin \mathcal{I}$ is reachable from $\mathcal{I}$ if and only if one of its neighbors is reachable from $\mathcal{I}$ in the graph $\RandomG\setminus\{i\}$, and the respective ingoing edge transmitted the contagion. Let $X_j^{-i}$ be a binary value indicating if $j$ is reachable from $\mathcal{I}$ in $\RandomG\setminus\{i\}$. Then
\begin{equation}
X_i = 0 \Leftrightarrow \forall j\in\OneToN\setminus\{i\}, X_j^{-i} = 0 \mbox{ or } A_{ji} = 0,
\end{equation}
which implies the following alternative expression for $X_i$:
\begin{equation}
1 - X_i = \prod_{j\neq i} (1 - X_j^{-i} A_{ji}).
\end{equation}

Moreover, the positive correlation of $\{1 - X_j^{-i} A_{ji}\}_{j\in \OneToN\setminus\{i\}}$ implies that
\begin{equation}
\Exp{\prod_{j\neq i} (1 - X_j^{-i} A_{ji})} \geq \prod_{j\neq i} \Exp{1 - X_j^{-i} A_{ji}}
\end{equation}
which leads to
\begin{equation}
\begin{array}{ll}
\Exp{X_i} &\leq 1 - \prod_{j\neq i} \Exp{1 - X_j^{-i} A_{ji}}\\
               &= 1 - \prod_{j\neq i} \left(1 - \Exp{X_j^{-i}}\Exp{A_{ji}}\right)\\
							 &\leq 1 - \prod_j \left(1 - \Exp{X_j}\Exp{A_{ji}}\right)\\
\end{array}
\end{equation}
since $X_j^{-i}$ and $A_{ji}$ are independent, due to hypothesis (H1) of the LPC property (see \Definition{def:LPC}) and $X_j^{-i}$ only depends on $(A_{kl})_{(k,l)\notin\mathcal{N}_{ji}}$.The second inequality comes from the fact that $X_j^{-i} \leq X_j \as$. Finally,
\begin{equation}
\begin{array}{ll}
\Exp{X_i} &\leq 1 - \exp \left(\sum_j \ln(1 - \Exp{X_j}\Exp{A_{ji}})\right)\\
               &\leq 1 - \exp \left(\sum_j \ln(1 - \Exp{A_{ji}})\Exp{X_j}\right)\\
               &= 1 - \exp \left(-\sum_j \HazMat_{ji} \Exp{X_j}\right)\\
\end{array}
\end{equation}
since we have on the one hand, for any $x \in [0,1]$ and $a < 1$, $\ln(1 - ax) \geq \ln(1 - a) x$, and on the other hand $\Exp{A_{ji}} = 1 - \exp(-\HazMat_{ji})$ by definition of $\HazMat$.
\end{proof}

Using \Lemma{lem:mainLemma}, we are now ready to start the proof of \Theorem{th:mainResult}.

\begin{proof}[Proof of \Theorem{th:mainResult}]

In order to simplify notations, we define $Z_i=\big(\Exp{X_i})_i$ that we collect in the vector $Z=(Z_i)_{i \in [1...n]}$. Using \Lemma{lem:mainLemma} and convexity of the exponential function, we have for any $u \in \mathbb{R}^n$ such that $\forall i\in \mathcal{I}, u_i = 0$ and $\forall i\notin \mathcal{I}, u_i \geq 0$,
\begin{equation}
\displaystyle u^\top Z \leq |u|_1 \bigg (1-\sum_{i=1}^{n-1} \frac {u_i}{|u|_1} \exp(-(\HazMat^\top Z)_i) \bigg)
         \leq |u|_1 \bigg (1- \exp \big(-\frac{Z^\top \HazMat u}{|u|_1} \big) \bigg)
\end{equation}
where $|u|_1 = \sum_i |u_i|$ is the $L_1$-norm of $u$.

Now taking $u=(1_{i\notin \mathcal{I}} Z_i)_i$ and noting that 	$\forall i, u_i \leq Z_i$, we have
\begin{equation}
\begin{array}{ll}
\displaystyle \frac{Z^\top Z-n_0}{|Z|_1-n_0}  \leq 1- \exp \bigg(-\frac{Z^\top \HazMat Z}{|Z|_1-n_0} \bigg)
 \leq 1- \exp \bigg(-\frac{\HazSpec (Z^\top Z-n_0)}{|Z|_1-n_0}-\frac{\HazSpec n_0}{|Z|_1-n_0} \bigg)
\end{array}
\end{equation}
where $\HazSpec=\rho(\frac{\HazMat+\HazMat^\top}{2})$. Defining $y=\frac{Z^\top Z-n_0}{|Z|_1-n_0}$ and $z=|Z|_1-n_0=\sigma(\mathcal{I})-n_0$, the aforementioned inequation rewrites
\begin{eqnarray}
y \leq 1- \exp \bigg(-\HazSpec y -\frac{\HazSpec n_0}{z} \bigg)
\end{eqnarray}
But by Cauchy-Schwarz inequality applied to $u$, $(n-n_0) (Z^\top Z-n_0) \geq (|Z|_1-n_0)^2$, which means that $z \leq y(n-n_0)$. We now consider the equation
\begin{eqnarray}\label{eqn:equationy}
x-1+\exp \bigg (-\HazSpec x-\frac {\HazSpec n_0}{x(n-n_0)}\bigg) =0
\end{eqnarray}
Because the function $f:x \rightarrow x-1+\exp \big (-\HazSpec x+\frac {\HazSpec n_0}{x(n-n_0)}\big)$ is continuous, verifies $f(1) > 0$ and $\lim_{x \rightarrow 0^+} f(x) = -1$, \Eq{eqn:equationy} admits a solution $\gamma_1$ in $]0,1[$.

We then prove by contradiction that $z \leq \gamma_1(n-n_0)$. Let us assume $z > \gamma_1(n-n_0)$. Then $y \leq 1- \exp \big(-\HazSpec y -\frac{\HazSpec n_0}{\gamma_1(n-n_0)} \big)$. But the function $h:x \rightarrow x-1+\exp \big (-\HazSpec x+\frac {\HazSpec n_0}{\gamma_1(n-n_0)}\big)$ is convex and verifies $h(0) < 0$ and $h(\gamma_1)=0$. Therefore, for any $y > \gamma_1$, $0=f(\gamma_1) \leq \frac{\gamma_1}{y}f(y)+(1-\frac{\gamma_1}{y})f(0)$, and therefore $f(y) > 0$. Thus, $y \leq \gamma_1$. But $z \leq y (n-n_0) \leq \gamma_1(n-n_0)$ which yields the contradiction.
\end{proof}

\begin{proof}[Proof of \Corollary{cor:simpleBounds}]
Using \Lemma{lem:gammaBounds} and observing that:
\[
\gamma_1 = \gamma\left(\HazSpec, \frac{\HazSpec n_0}{\gamma_1(n-n_0)}\right) \leq \gamma\left(\HazSpec, \frac{\HazSpec n_0}{(\gamma_1-\gamma_0(\HazSpec))(n-n_0)}\right)~,
 \]
we obtain the following bounds:
\begin{equation}
\gamma_1 \leq \gamma_0(\HazSpec) + \frac{\HazSpec n_0 (1 - \gamma_0(\HazSpec))}{(\gamma_1-\gamma_0(\HazSpec))(n-n_0)(1 - \HazSpec(1-\gamma_0(\HazSpec)))},
\end{equation}
and
\begin{equation}
\gamma_1 \leq \gamma_0(\HazSpec) + \sqrt{\frac{2n_0}{(\gamma_1-\gamma_0(\HazSpec))(n-n_0)}},
\end{equation}
which lead to
\begin{equation}\label{eq:subsupBound}
\gamma_1 \leq \gamma_0(\HazSpec) + \sqrt{\frac{\HazSpec (1 - \gamma_0(\HazSpec))}{1 - \HazSpec(1-\gamma_0(\HazSpec))}}\sqrt{\frac{n_0}{n-n_0}},
\end{equation}
and
\begin{equation}\label{eq:criticalBound}
\gamma_1 \leq \gamma_0(\HazSpec) + \left(\frac{2n_0}{n-n_0}\right)^{1/3}.
\end{equation}
The subcritical and supercritical regimes are obtained using \Eq{eq:subsupBound} (recall that $\gamma_0(\rho) = 0$ when $\rho \leq 0$) and the critical regime using \Eq{eq:criticalBound} and \Lemma{lem:gammaSlope}.
\end{proof}

\subsection{Proofs of \Theorem{th:uniformResult} and \Corollary{cor:uniformsimpleBounds}}

In this subsection, we develop the proofs for \Theorem{th:uniformResult} and \Corollary{cor:uniformsimpleBounds} in the case when the set of influencers $\mathcal{I}$ is drawn from a uniform distribution over $\mathcal{P}_{n_0}(\OneToN)$.

We start with an important lemma that will play the same role in the proof of \Theorem{th:uniformResult} than \Lemma{lem:mainLemma} in the proof of \Theorem{th:mainResult}.

\begin{lemma}\label{lem:mainLemmaUniform}

Assume $\mathcal{I}$ is drawn from a uniform distribution over $\mathcal{P}_{n_0}(\OneToN)$. Then, for any $i \in \OneToN$, the probability $\Exp{X_i}$  that node $i$ is reachable from $\mathcal{I}$ in $\RandomG$ satisfies the following implicit inequation:

\begin{equation}
\Exp{X_i} \leq 1 - \frac{n-n_0}{n} \exp \bigg(- \frac{n}{n-n_0} \sum_j \HazMat_{ji} \Exp{X_j} \bigg)
\end{equation}
\end{lemma}

\begin{proof}
\begin{equation}
\begin{array}{ll}
\Exp{X_i} &\displaystyle =\Exp{1_{\{i \in \mathcal{I}\}}}+\Exp{1_{\{i \notin \mathcal{I}\}}}\Exp{\Exp{X_i|\mathcal{I}}|i \notin \mathcal{I}} \\
            &\displaystyle \leq \frac{n_0}{n} + \frac{n-n_0}{n} \bigg(1 - \Exp{\exp \left(-\sum_j \HazMat_{ji} \Exp{X_j|\mathcal{I}} \right)|i \notin \mathcal{I}} \bigg)\\
            &\displaystyle \leq \frac{n_0}{n} + \frac{n-n_0}{n} \bigg(1 - \exp \left(-\Exp{\sum_j \HazMat_{ji} \Exp{X_j|\mathcal{I}} |i \notin \mathcal{I}}\right) \bigg)\\
            &\displaystyle = 1 - \frac{n-n_0}{n} \exp \left(-\sum_j \HazMat_{ji} \Exp{X_j|i \notin \mathcal{I}} \right)\\
            &\displaystyle \leq 1 - \frac{n-n_0}{n} \exp \left(-\frac{n}{n-n_0} \sum_j \HazMat_{ji} \Exp{X_j} \right)
\end{array}
\end{equation}
where the first inequality is \Lemma{lem:mainLemma} and the second one is Jensen inequality for conditional expectations.
\end{proof}
\begin{proof}[Proof of \Theorem{th:uniformResult}]
We define $Z_i=\big(\Exp{X_i})_i$ that we collect in the vector $Z=(Z_i)_{i \in [1...n]}$. Then, using \Lemma{lem:mainLemmaUniform}, and convexity of exponential function, we have:
\begin{equation}
\begin{array}{ll}
\displaystyle \frac{Z^\top Z}{|Z|_1} & \displaystyle \leq 1- \frac{n-n_0}{n} \sum_{i=1}^n \frac {Z_i}{|Z|_1} \exp \big(-\frac{n}{n-n_0} (\HazMat^\top Z)_i\big)\\
                       &\displaystyle\leq 1- \frac{n-n_0}{n} \exp \big(-\frac{n}{n-n_0} \frac{Z^\top \HazMat Z}{|Z|_1} \big)\\
											 & \displaystyle \leq 1- \frac{n-n_0}{n} \exp \big(- \frac{n\HazSpec}{n-n_0}\frac{Z^\top Z}{|Z|_1} \big),
\end{array}
\end{equation}
which leads, due to the monotonicity of $x \mapsto 1-\frac{n-n_0}{n}\exp(-\frac{n\HazSpec}{n-n_0} x)$, to
\begin{equation}
\frac{Z^\top Z}{|Z|_1} \leq \gamma\left(\frac{n\HazSpec}{n-n_0}, -\ln(1 - \frac{n_0}{n})\right) = (1 - \frac{n_0}{n})\gamma\left(\HazSpec, \frac{\HazSpec n_0}{n-n_0}\right) + \frac{n_0}{n}.
\end{equation}
Finally, we have by Cauchy-Schwarz inequality $\sigma_U = |Z|_1 \leq n \frac{Z^\top Z}{|Z|_1}$, which proves the proposition.
\end{proof}

\begin{proof}[Proof of \Corollary{cor:uniformsimpleBounds}]
Using \Lemma{lem:gammaBounds}, we obtain the following bounds:
\begin{equation}\label{eq:subsupBoundUniform}
\gamma\left(\HazSpec, \frac{\HazSpec n_0}{n-n_0}\right) \leq \gamma_0(\HazSpec) + \frac{\HazSpec n_0 (1 - \gamma_0(\HazSpec))}{(n-n_0)(1 - \HazSpec(1-\gamma_0(\HazSpec)))},
\end{equation}
and
\begin{equation}\label{eq:criticalBoundUniform}
\gamma(\HazSpec, \frac{\HazSpec n_0}{n-n_0}) \leq \gamma_0(\HazSpec) + \sqrt{\frac{2n_0}{n-n_0}},
\end{equation}
The subcritical and supercritical regimes are obtained using \Eq{eq:subsupBoundUniform} (recall that $\gamma_0(\rho) = 0$ when $\rho \leq 0$) and the critical regime using \Eq{eq:criticalBoundUniform} and \Lemma{lem:gammaSlope}.
\end{proof}

\subsection{Proofs of \Theorem{th:randomResult} and \Corollary{cor:randomsimpleBounds}}

In this subsection, we develop the proofs for \Theorem{th:randomResult} and \Corollary{cor:randomsimpleBounds} in the case when each node belongs to the set of influencers $\mathcal{I}$ independently at random with probability $q$.

We start with an important lemma that will play the same role in the proof of \Theorem{th:randomResult} than \Lemma{lem:mainLemma} in the proof of \Theorem{th:mainResult}.

\begin{lemma}\label{lem:mainLemmaRandom}

Assume each node is an influencer with independent probability $q\in[0,1]$ and denote by $\mathcal{I}$ the random set of influencers that is drawn. Then, for any $i \in \OneToN$, the probability $\Exp{X_i}$  that node $i$ is reachable from $\mathcal{I}$ in $\RandomG$ satisfies the following implicit inequation:

\begin{equation}
\Exp{X_i} \leq 1 - (1-q) \exp \bigg(- \sum_j \HazMat_{ji} \Exp{X_j} \bigg)
\end{equation}
\end{lemma}

\begin{proof}
\begin{equation}
\begin{array}{ll}
\Exp{X_i} &\displaystyle =\Exp{1_{\{i \in \mathcal{I}\}}}+\Exp{1_{\{i \notin \mathcal{I}\}}}\Exp{\Exp{X_i|\mathcal{I}}|i \notin \mathcal{I}} \\
            &\displaystyle \leq q + (1-q) \bigg(1 - \Exp{\exp \left(-\sum_j \HazMat_{ji} \Exp{X_j|\mathcal{I}} \right)|i \notin \mathcal{I}} \bigg)\\
            &\displaystyle \leq q + (1-q) \bigg(1 - \exp \left(-\Exp{\sum_j \HazMat_{ji} \Exp{X_j|\mathcal{I}} |i \notin \mathcal{I}}\right) \bigg)\\
            & \displaystyle  = 1 - (1-q) \exp \left(-\sum_j \HazMat_{ji} \Exp{X_j|i \notin \mathcal{I}} \right)\\
            & \displaystyle  \leq 1 - (1-q) \exp \left(- \sum_j \HazMat_{ji} \Exp{X_j} \right)
\end{array}
\end{equation}
where the first inequality is \Lemma{lem:mainLemma}, the second one is Jensen's inequality for conditional expectations, and the third is the positive correlation of $X_j$ and $\one\{i\in\mathcal{I}\}$.
\end{proof}
\begin{proof}[Proof of \Theorem{th:randomResult}]
We define $Z_i=\big(\Exp{X_i})_i$ that we collect in the vector $Z=(Z_i)_{i \in [1...n]}$. Then, using \Lemma{lem:mainLemmaRandom}, and convexity of exponential function, we have:
\begin{equation}
\begin{array}{ll}
\displaystyle \frac{Z^\top Z}{|Z|_1} &\displaystyle \leq 1- (1-q) \sum_{i=1}^n \frac {Z_i}{|Z|_1} \exp \big(- (\HazMat^\top Z)_i\big)\\
                       & \displaystyle \leq 1- (1-q) \exp \big(- \frac{Z^\top \HazMat Z}{|Z|_1} \big)\\
											 & \displaystyle \leq 1- (1-q) \exp \big(- \HazSpec\frac{Z^\top Z}{|Z|_1} \big),
\end{array}
\end{equation}
which leads, due to the monotonicity of $x \mapsto 1-(1-q)\exp(-\HazSpec x)$, to
\begin{equation}
\frac{Z^\top Z}{|Z|_1} \leq \gamma(\HazSpec, -\ln(1-q)).
\end{equation}
Finally, we have by Cauchy-Schwarz inequality $\sigma_R = |Z|_1 \leq n \frac{Z^\top Z}{|Z|_1}$, which proves the proposition.
\end{proof}
\begin{proof}[Proof of \Corollary{cor:uniformsimpleBounds}]
Using \Lemma{lem:gammaBounds}, we obtain the following bounds:
\begin{equation}\label{eq:subsupBoundRandom}
\gamma(\HazSpec, -\ln(1-q)) \leq \gamma_0(\HazSpec) + \frac{-\ln(1-q) (1 - \gamma_0(\HazSpec))}{1 - \HazSpec(1-\gamma_0(\HazSpec))},
\end{equation}
and
\begin{equation}\label{eq:criticalBoundRandom}
\gamma(\HazSpec, -\ln(1-q)) \leq \gamma_0(\HazSpec) + \sqrt{-2\ln(1-q)},
\end{equation}
The subcritical and supercritical regimes are obtained using \Eq{eq:subsupBoundRandom} (recall that $\gamma_0(\rho) = 0$ when $\rho \leq 0$) and the critical regime using \Eq{eq:criticalBoundRandom} and \Lemma{lem:gammaSlope}.
\end{proof}

\subsection{Proof of \Proposition{prop:tightness}}

\begin{proof}[Proof of \Proposition{prop:tightness}]
Let $\RandomG_{a,b} = \RandomG(n,A)$ be a ``\emph{random star-network}'', \ie an undirected random graph such that $\{A_{ij} : i<j\}$ are independent Bernoulli random variables of parameter $a\in[0,1]$ if $i = 1$ and $b<a$ otherwise. Then, the next Lemma shows that the influence of node $1$ in $\RandomG_{a,b}$ is lower bounded by the size of the giant component of an \Erdos graph.
\begin{lemma}
Let $\RandomG_{a,b} = \RandomG(n,A)$ be a ``\emph{random star-network}'' of parameters $a$ and $b$, and $\RandomG(n,p)$ an \Erdos graph of size $n$ and parameter $p$. The influence $\sigma_{a,b}(\{1\})$ of node $1$ in $\RandomG_{a,b}$ is lower bounded by
\begin{equation}
\sigma_{a,b}(\{1\}) \geq 1 - \frac{1}{ae} + na + (1-a)\Exp{C_1(\RandomG(n-1,b))},
\end{equation}
where $C_1(\RandomG)$ denotes the size of the giant component of $\RandomG$.
\end{lemma}
\begin{proof}
Since the edge presence variables $A_{ij}$ are independent, the set of nodes linked to $1$ in $\RandomG_{a,b}$ is a random set $\mathcal{I}(a)$ such that each node in $\{2,...,n\}$ belongs to it independently with probability $a$. Also, $\mathcal{I}(a)$ is independent from the subgraph restricted to $\{2,...,n\}$, and since each edge in $\{2,...,n\}$ is drawn independently and has probability $b$, this subgraph is an \Erdos graph of size $n-1$ and parameter $b$. Hence, if $\Exp{\sigma_b(\mathcal{I}(a))}$ is the influence of a random set $\mathcal{I}(a)$ in $\RandomG(n-1,b)$ as defined in \Theorem{th:randomResult}, then
\begin{equation}
\sigma_{a,b}(\{1\}) = 1 + \Exp{\sigma_b(\mathcal{I}(a))}.
\end{equation}
Hence, \Eq{eq:linkperco} and the same derivation as in \Theorem{th:sizecomponent} gives that:
\begin{equation}
\begin{array}{ll}
\sigma_{a,b}(\{1\}) & \displaystyle \geq 1 + na + (1-a)\Exp{C_1(\RandomG(n-1,b)) (1 - (1-a)^{C_1(\RandomG(n-1,b))-1})}\\
&\displaystyle \geq 1 + na + (1-a)\Exp{C_1(\RandomG(n-1,b))} - \frac{1}{-\ln(1-a)e}\\
& \displaystyle \geq 1 + na + (1-a)\Exp{C_1(\RandomG(n-1,b))} - \frac{1}{ae}.
\end{array}
\end{equation}
\end{proof}
However, a simple calculation holds $\HazSpec = \frac{(n-2)b' + \sqrt{(n-2)^2b'^2 + 4 (n-1) a'^2}}{2}$, where $a' = -\ln(1-a)$ and $b'=-\ln(1-b)$. We now conclude in the three regimes:

\paragraph{Subcritical regime} ($\HazSpec < 1$): In this case, we take $b = 0$ and $a = \frac{\rho}{\sqrt{n-1}}$, for $\rho\in[0,1)$. Then $\HazSpec = \rho + O(\frac{1}{\sqrt{n}})$ and $\sigma_{a,b} = 1 + \rho \sqrt{n-1} \geq \frac{\rho}{2}\sqrt{n}$ for $n$ sufficiently large.

\paragraph{Critical and supercritical regime} ($\HazSpec \geq 1$): In this case, we take $a = \frac{1}{\sqrt{n\ln{n}}}$ and $b = \frac{\rho}{n}$. Then $\HazSpec = \rho + O(\frac{1}{\ln{n}})$ and
\begin{equation}
\sigma_{a,b} \geq O(\sqrt{n\ln{n}}) + \Exp{C_1(\RandomG(n-1,b))}.
\end{equation}
However, classical results in percolation theory \cite{erd6s1960evolution} state that, for $\eta > 0$ and $\omega(n)$ any function \st $\lim_{n\rightarrow +\infty}\omega(n) = +\infty$,
\begin{equation}
C_1(\RandomG(n,\frac{1}{n})) \geq \frac{n^{2/3}}{\omega(n)} \aas
\end{equation}
and
\begin{equation}
\left|\frac{C_1(\RandomG(n,\frac{\rho}{n}))}{n} - \gamma_0(c)\right| \leq \eta \aas
\end{equation}
Hence, in the first case, Markov's inequality gives $\Exp{C_1(\RandomG(n-1,b))} \geq \frac{n^{2/3}}{\omega(n)}(1 - o(1))$, which leads to, for $n$ sufficiently large,
\begin{equation}
\Exp{C_1(\RandomG(n-1,b))} \geq C_{\rho}n^{2/3},
\end{equation}
for some $C_\rho > 0$, since assuming $\liminf_{n\rightarrow +\infty} \frac{\Exp{C_1(\RandomG(n-1,b))}}{n^{2/3}} = 0$ and taking $\omega(n) = \max_{m\leq n} \sqrt{\frac{n^{2/3}}{\Exp{C_1(\RandomG(n-1,b))}}}$ leads to $\liminf_{n\rightarrow +\infty} \sqrt{\frac{\Exp{C_1(\RandomG(n-1,b))}}{n^{2/3}}} \geq 1$, which contradicts the assumption. For the second case, Markov's inequality gives
\begin{equation}
\begin{array}{ll}
\Exp{C_1(\RandomG(n-1,b))} &\geq (\gamma_0(\rho)n - \eta n)\Prob{C_1(\RandomG(n-1,b)) \geq \gamma_0(\rho)n - \eta n}\\
&\geq \gamma_0(\rho)n - \eta n - o(n).
\end{array}
\end{equation}
Taking the limit inferior leads to, for all $\eta > 0$:
\begin{equation}
\liminf_{n\rightarrow +\infty}\frac{\Exp{C_1(\RandomG(n-1,b))} - \gamma_0(\rho)n}{n} \geq \liminf_{n\rightarrow +\infty} \{- \eta - o(1)\} = -\eta~,
\end{equation}
and thus $\liminf_{n\rightarrow +\infty}\frac{\Exp{C_1(\RandomG(n-1,b))} - \gamma_0(\rho)n}{n} \geq 0$, which can be rewritten as $\Exp{C_1(\RandomG(n-1,b))} \geq \gamma_0(\rho)n - o(n)$.
\end{proof}

\section{Proofs of the percolation theorems}

The aim of this section is to prove the results obtained in section \ref{sec:percolation} for the bond percolation problem from the general results on reachablity sets of section \ref{sec:bounds}. We recall that we consider an undirected random graph $\RandomG(n,A)$ of size $n$ with independent edge presence variables $\{A_{ij} : i < j\}$, and denote by $C_k(\RandomG)$ the size of its $k^{th}$-largest connected component, as well as $N(m)$ the number of connected components of $\RandomG$ of cardinality greater than or equal to $m$. We also recall that we are able to relate the distribution of the sizes of connected components of $\RandomG$ to the Hazard function through equation \ref{eq:linkperco}. Let $a>0$, then:

\begin{equation*}
\Exp{~\sum_k {C_k(\RandomG)\left(1-e^{-a C_k(\RandomG)}\right)}~}\leq \gamma(\HazSpec, a) n
\end{equation*}

\subsection{Proofs of \Theorem{th:sizecomponent} and \Corollary{cor:simpleSize}}
\begin{proof}[Proof of \Theorem{th:sizecomponent}]
\Theorem{th:sizecomponent} is simply obtained by combining equation \ref{eq:linkperco} and the following observation:
\begin{equation*}
\sum_k {C_k(\RandomG)\left(1-e^{-a C_k(\RandomG)}\right)} \leq C_1(\RandomG)\left(1-e^{-a C_1(\RandomG)}\right) + (n-C_1(\RandomG)) (1-e^{-a})
\end{equation*}
Therefore,
\begin{equation*}
\Exp{C_1(\RandomG)(1-e^{-a (C_1(\RandomG)-1)})} \leq n e^a \left(\gamma (\HazSpec,a)-1+e^{-a}\right) = n \left(1-e^{-\HazSpec\gamma (\HazSpec,a)} \right).
\end{equation*}
\end{proof}

\begin{proof}[Proof of \Corollary{cor:simpleSize}]
We first prove the subcritical result. Let $a \geq 0$. When $\HazSpec<1$, we have $\gamma (\HazSpec,a)=0$ and therefore \Lemma{lem:gammaBounds} implies $\gamma (\HazSpec,a) \leq \frac{a}{1-\HazSpec}$. By convexity of exponential function, we get:
\begin{equation*}
\Exp{C_1(\RandomG)(1-e^{-a (C_1(\RandomG)-1)})} \leq \frac{n a \HazSpec} {1-\HazSpec}.
\end{equation*}
This inequality between two derivable functions of $a$ such that $f(a) \leq g(a)$ for all $a \geq 0$ and $f(0)=g(0)$ implies that $\frac{\partial f}{\partial a}(0) \leq \frac{\partial g}{\partial a}(0)$ which yields:
\begin{equation*}
\Exp{C_1(\RandomG)(C_1(\RandomG)-1)} \leq \frac{n \HazSpec} {1-\HazSpec}.
\end{equation*}
The first equation of \Corollary{cor:simpleSize} is then a straightforward resolution of a second-order equation, using the fact that $\Exp{C_1(\RandomG)^2} \geq  \Exp{C_1(\RandomG)}^2$.

\medskip
For the critical and supercritical results, we will make use of the fact that, for all $a>0$, $C_1(\RandomG) e^{-a C_1(\RandomG)} \leq \frac{1}{ae}$, which yields:
\begin{equation}\label{eq:supercriticalapprox}
\Exp{C_1(\RandomG)} \leq \frac{e^a}{ae} + n \left(1-e^{-\HazSpec\gamma (\HazSpec,a)} \right)
\end{equation}
which rewrites $\Exp{C_1(\RandomG)} + (n-\Exp{C_1(\RandomG)})(1-e^{-a}) \leq \frac{1}{ae} + n \gamma (\HazSpec,a)$ and therefore implies:
\begin{equation}\label{eq:criticalapprox}
\Exp{C_1(\RandomG)} \leq \frac{1}{ae} + n \gamma (\HazSpec,a)
\end{equation}
From \Lemma{lem:gammaBounds}, we know that for $\HazSpec \neq 0$, $\gamma (\HazSpec,a) \leq \gamma_0(\HazSpec) + \sqrt{2a}$. We therefore get the critical result using equation \ref{eq:criticalapprox}:

\begin{equation*}
\Exp{C_1(\RandomG)} \leq n \gamma_0(\HazSpec) + \min_{a>0} \left \{\frac{1}{ae} +  n \sqrt{2a}\right\} = n \gamma_0(\HazSpec)+ n^{2/3} \left(\frac{27}{2e}\right)^{1/3}.
\end{equation*}
For the supercritical result, we use equation  \ref{eq:supercriticalapprox} and the fact that
\begin{align*}
1-e^{-\HazSpec\gamma (\HazSpec,a)} & = 1-e^{-\HazSpec\gamma_0(\HazSpec)}\left(1-e^{-\HazSpec(\gamma(\HazSpec,a)-\gamma_0(\HazSpec))} \right) \\ \nonumber
& \leq \gamma_0(\HazSpec) + \HazSpec\left(1-\gamma_0(\HazSpec)\right)\left(\gamma(\HazSpec,a)-\gamma_0(\HazSpec)\right) \\ \nonumber
& \leq \gamma_0(\HazSpec) + \frac{a \HazSpec\left(1-\gamma_0(\HazSpec)\right)^2}{1-\HazSpec+\HazSpec\gamma_0(\HazSpec)}.
\end{align*}
We then choose
\begin{equation*}
a = \sqrt{\frac{2(1-\HazSpec+\HazSpec\gamma_0(\HazSpec))}{2e \HazSpec\left(1-\gamma_0(\HazSpec)\right)^2 n +1-\HazSpec+\HazSpec\gamma_0(\HazSpec)}}
\end{equation*}
which gives us
\begin{equation*}
\Exp{C_1(\RandomG)} \leq n \gamma_0(\HazSpec) + \frac{2}{e} \sqrt{\frac{e n \HazSpec (1-\gamma_0(\HazSpec))^{2} }{1-\HazSpec+\HazSpec\gamma_0(\HazSpec)}+\frac{1}{2}} + \frac{1}{e} + \frac{e^a-1}{ae}~,
\end{equation*}
Using the fact that $a < \sqrt{2}$ and $\sqrt{x+y} \leq \sqrt{x} + \sqrt{y}$, we finally get:
\begin{equation*}
\Exp{C_1(\RandomG)} \leq n \gamma_0(\HazSpec) + \frac{2}{\sqrt{e}} \sqrt{\frac{n \HazSpec (1-\gamma_0(\HazSpec))^{2} }{1-\HazSpec+\HazSpec\gamma_0(\HazSpec)}} + \frac{1+\sqrt{2}+e^{\sqrt{2}}}{\sqrt{2}e}~,
\end{equation*}
which yields the supercritical result.
\end{proof}

\subsection{Proofs of \Theorem{th:numcomponents} and \Corollary{cor:simplenumcomponents}}

\begin{proof}[Proof of \Theorem{th:numcomponents}]

Let $a>0$. In order to prove \Theorem{th:numcomponents}, we start again from equation \ref{eq:linkperco} and use the fact that:
\begin{equation*}
\sum_k {C_k(\RandomG)\left(1-e^{- a C_k(\RandomG)}\right)\one\{C_k(\RandomG)\geq m}\} \leq  (1-e^{- a m}) \sum_{k} {C_k(\RandomG)\one\{C_k(\RandomG)\geq m}\} \mbox { and}
\end{equation*}
\begin{equation*}
\sum_k {C_k(\RandomG)\left(1-e^{- a C_k(\RandomG)}\right)\one\{C_k(\RandomG)< m}\} \leq (1-e^{- a})\left(n-\sum_{k} {C_k(\RandomG)\one\{C_k(\RandomG)\geq m}\}\right).
\end{equation*}
Therefore, we have:
\begin{equation*}
\sum_{k} {C_k(\RandomG)\one\{C_k(\RandomG)\geq m}\} \leq \frac{n e^a \left(\gamma (\HazSpec,a)-1+e^{-a}\right)}{1-e^{-a(m-1)}} = \frac{n \left(1-e^{-\HazSpec\gamma (\HazSpec,a)} \right)}{1-e^{-a(m-1)}}.
\end{equation*}
which proves the theorem, noting that $m N(m) \leq \sum_{k} {C_k(\RandomG)\one\{C_k(\RandomG)\geq m}\}$
\end{proof}
\begin{proof}[Proof of \Corollary{cor:simplenumcomponents}]
In the subcritical case, we have $\gamma (\HazSpec,a) \leq \frac{a}{1-\HazSpec}$ which means that, for all $a>0$, $\HazSpec<1$:
\begin{equation*}
N(m) \leq \frac{n}{m} \frac{a \HazSpec}{(1-\HazSpec)(1-e^{-a(m-1)})}
\end{equation*}
The right-hand side function of $a$ is increasing on the semi-line, and we therefore takes its limit when $a \rightarrow 0$ to get the subcritical result.

\medskip

For the critical case, we note that \Theorem{th:numcomponents} implies that, for all $a>0$:
\begin{equation*}
(1-e^{-am })m N(m) + (1-e^{-a}) (n-m N(m)) \leq n \gamma (\HazSpec,a)
\end{equation*}
and therefore
\begin{equation*}
N(m) \leq \frac{n}{m} \frac{\gamma (\HazSpec,a)}{1-e^{-am}} \leq \frac{n}{m} \frac{\gamma_0(\HazSpec) + \sqrt{2a}}{1-e^{-am}} = \frac{n}{m^{3/2}} \frac{\gamma_0(\HazSpec)\sqrt{m} + \sqrt{2am}}{1-e^{-am}}.
\end{equation*}
The function $x \mapsto \frac{\sqrt{2x}}{1-e^{-x}}$ admits a unique minimum for $x>0$ in $\eta$ which is the strictly positive solution of $e^\eta = 2 \eta + 1$. Setting $a = \frac{\eta}{m}$ yields:
\begin{equation*}
N(m) \leq \frac{n}{m^{3/2}} \frac{\gamma_0(\HazSpec)\sqrt{m} + \sqrt{2 \eta}}{1-e^{-\eta}}
\end{equation*}
Hence, if $\HazSpec \leq 1 - \nu m^{-1/2}$ for a fixed $\nu > 0$, \Lemma{lem:gammaSlope} implies:
\begin{equation*}
N(m) \leq \frac{n}{m^{3/2}} \frac{2\nu + \sqrt{2 \eta}}{1-e^{-\eta}}.
\end{equation*}
The critical result is given by finding the value $\nu$ for which the first orders of the subcritical and critical bounds are equal at the threshold value $\rho = 1 - \nu m^{-1/2}$, \ie $\nu$ is the solution of $\frac{1}{\nu} = \frac{2\nu + \sqrt{2 \eta}}{1-e^{-\eta}}$.

\medskip

For the supercritical result, we will make use of the fact that
\begin{equation*}
1-e^{-\HazSpec\gamma (\HazSpec,a)} \leq \gamma_0(\HazSpec) + \frac{a \HazSpec\left(1-\gamma_0(\HazSpec)\right)^2}{1-\HazSpec+\HazSpec\gamma_0(\HazSpec)}.
\end{equation*}
Introducing $B=\frac{\HazSpec\left(1-\gamma_0(\HazSpec)\right)^2}{1-\HazSpec+\HazSpec\gamma_0(\HazSpec)}$, \Theorem{th:numcomponents} gives:
\begin{equation}\label{eq:supercriticalwithB}
N(m) \leq \frac{n}{m} \left(\frac{\gamma_0(\HazSpec) + B a}{1-e^{a(m-1)}}\right)
\end{equation}
Derivating the right-hand side with respect to $a$ and setting $x=a(m-1)$, we find that the minimizer $x^{\star}$ is given by the unique strictly positive solution of $e^x = 1 + x + \frac{B}{\gamma_0(\HazSpec)}(m-1)$. Therefore, we now in particular that $x^{\star} \leq \sqrt{2(e^{x^\star}-1-x^{\star})} = \sqrt{2 \gamma_0(\HazSpec) (m-1) / B}$. The supercritical result is obtained by plugging
\begin{equation*}
a = \sqrt{\frac{2 \gamma_0(\HazSpec)}{B(m-1)}}
\end{equation*}
into equation \ref{eq:supercriticalwithB}.
\end{proof}

\section*{Acknowledgements}
This research is part of the SODATECH project funded by the French Government within the program of ``\emph{Investments for the Future\,--\,Big Data}''.

\bibliographystyle{siam}
\bibliography{SpectralBounds}

\begin{thebibliography}{10}

\bibitem{PhysRevE.69.050901}
{\sc E.~Ben-Naim and P.~L. Krapivsky}, {\em Size of outbreaks near the epidemic
  threshold}, Physical Review E, 69 (2004), p.~050901.

\bibitem{EJP817}
{\sc S.~Bhamidi, R.~van~der Hofstad, and J.~van Leeuwaarden}, {\em Scaling
  limits for critical inhomogeneous random graphs with finite third moments},
  Electronic Journal of Probability, 15 (2010), pp.~1682--1702.

\bibitem{bollobas1984evolution}
{\sc B.~Bollob{\'a}s}, {\em The evolution of random graphs}, Transactions of
  the American Mathematical Society, 286 (1984), pp.~257--274.

\bibitem{bollobas2010percolation}
{\sc B.~Bollob{\'a}s, C.~Borgs, J.~Chayes, and O.~Riordan}, {\em Percolation on
  dense graph sequences}, The Annals of Probability, 38 (2010), pp.~150--183.

\bibitem{bollobas2007phase}
{\sc B.~Bollob{\'a}s, S.~Janson, and O.~Riordan}, {\em The phase transition in
  inhomogeneous random graphs}, Random Structures \& Algorithms, 31 (2007),
  pp.~3--122.

\bibitem{chen2009efficient}
{\sc W.~Chen, Y.~Wang, and S.~Yang}, {\em Efficient influence maximization in
  social networks}, in Proceedings of the 15th ACM SIGKDD International
  Conference on Knowledge Discovery and Data Mining, ACM, 2009, pp.~199--208.

\bibitem{draief2008}
{\sc M.~Draief, A.~Ganesh, and L.~Massouli\'{e}}, {\em Thresholds for virus
  spread on networks}, The Annals of Applied Probability, 18 (2008),
  pp.~359--378.

\bibitem{erd6s1960evolution}
{\sc P.~Erd{\"o}s and A.~R{\'e}nyi}, {\em On the evolution of random graphs},
  Publications of the Mathematical Institute of the Hungarian Academy of
  Sciences, 5 (1960), pp.~17--61.

\bibitem{Georgii99therandom}
{\sc H.~Georgii, O.~H\"aggstr\"om, and C.~Maes}, {\em The random geometry of
  equilibrium phases}, Phase Transitions and Critical Phenomena, 18 (1999),
  pp.~1--142.

\bibitem{DBLP:conf/icml/Gomez-RodriguezBS11}
{\sc M.~Gomez-Rodriguez, D.~Balduzzi, and B.~Sch{\"o}lkopf}, {\em Uncovering
  the temporal dynamics of diffusion networks}, in Proceedings of the 29th
  International Conference on Machine Learning, 2011, pp.~561--568.

\bibitem{rodriguez2012influence}
{\sc M.~Gomez-Rodriguez and B.~Sch{\"o}lkopf}, {\em Influence maximization in
  continuous time diffusion networks}, in Proceedings of the 29th International
  Conference on Machine Learning, 2012, pp.~313--320.

\bibitem{grassberger2003critical}
{\sc P.~Grassberger}, {\em Critical percolation in high dimensions}, Physical
  Review E, 67 (2003), p.~036101.

\bibitem{holley1974}
{\sc R.~Holley}, {\em Remarks on the fkg inequalities}, Communications in
  Mathematical Physics, 36 (1974), pp.~227--231.

\bibitem{Kempe:2003:MSI:956750.956769}
{\sc D.~Kempe, J.~Kleinberg, and E.~Tardos}, {\em Maximizing the spread of
  influence through a social network}, in Proceedings of the 9th ACM SIGKDD
  International Conference on Knowledge Discovery and Data Mining, ACM, 2003,
  pp.~137--146.

\bibitem{kermack1932contributions}
{\sc W.~O. Kermack and A.~G. McKendrick}, {\em Contributions to the
  mathematical theory of epidemics. ii. the problem of endemicity}, Proceedings
  of the Royal society of London. Series A, 138 (1932), pp.~55--83.

\bibitem{opac-b1130933}
{\sc E.~D. Kolaczyk}, {\em Statistical analysis of network data : methods and
  models}, Springer series in statistics, Springer, New York, NY, USA, 2009.

\bibitem{NIPS2014_5364}
{\sc R.~Lemonnier, K.~Scaman, and N.~Vayatis}, {\em Tight bounds for influence
  in diffusion networks and application to bond percolation and epidemiology},
  in Advances in Neural Information Processing Systems, 2014, pp.~846--854.

\bibitem{luczak1990component}
{\sc T.~{\L}uczak}, {\em Component behavior near the critical point of the
  random graph process}, Random Structures \& Algorithms, 1 (1990),
  pp.~287--310.

\bibitem{molloy1995critical}
{\sc M.~Molloy and B.~Reed}, {\em A critical point for random graphs with a
  given degree sequence}, Random structures \& algorithms, 6 (1995),
  pp.~161--179.

\bibitem{molloy1998size}
\leavevmode\vrule height 2pt depth -1.6pt width 23pt, {\em The size of the
  giant component of a random graph with a given degree sequence},
  Combinatorics probability and computing, 7 (1998), pp.~295--305.

\bibitem{nelson2007epidemiology}
{\sc K.~E. Nelson}, {\em Epidemiology of infectious disease: general
  principles}, Infectious Disease Epidemiology Theory and Practice.
  Gaithersburg, MD: Aspen Publishers,  (2007), pp.~17--48.

\bibitem{Newman:2010:NI}
{\sc M.~Newman}, {\em Networks: An Introduction}, Oxford University Press,
  Inc., New York, NY, USA, 2010.

\bibitem{norros2006}
{\sc I.~Norros and H.~Reittu}, {\em On a conditionally poissonian graph
  process}, Advances in Applied Probability, 38 (2006), pp.~59--75.

\bibitem{prakash2012threshold}
{\sc B.~A. Prakash, D.~Chakrabarti, N.~C. Valler, M.~Faloutsos, and
  C.~Faloutsos}, {\em Threshold conditions for arbitrary cascade models on
  arbitrary networks}, Knowledge and Information Systems, 33 (2012),
  pp.~549--575.

\bibitem{NIPS2015_5701}
{\sc K.~Scaman, R.~Lemonnier, and N.~Vayatis}, {\em Anytime influence bounds
  and the explosive behavior of continuous-time diffusion networks}, in
  Advances in Neural Information Processing Systems, 2015, pp.~2017--2025.

\bibitem{van2009virus}
{\sc P.~Van~Mieghem, J.~Omic, and R.~Kooij}, {\em Virus spread in networks},
  IEEE/ACM Transactions on Networking, 17 (2009), pp.~1--14.

\end{thebibliography}
\end{document}